\documentclass[12pt,oneside,a4paper]{amsart} %english
\usepackage[]{fontenc}
\usepackage[latin1]{inputenc}
\usepackage{amssymb,amsmath}

 \theoremstyle{plain}
\newtheorem{thm}{Theorem}[section]
  \theoremstyle{plain}
  \newtheorem{prop}[thm]{Proposition}
  \theoremstyle{plain}
  \newtheorem{cor}[thm]{Corollary}
  \theoremstyle{plain}
  \newtheorem{lem}[thm]{Lemma}
\theoremstyle{plain}
  \newtheorem{rem}[thm]{Remark}
\theoremstyle{plain}
  \newtheorem{conj}[thm]{Conjecture}
\theoremstyle{plain}
  \newtheorem{defn}[thm]{Definition}

\usepackage{amsfonts}

\usepackage{amscd}

\usepackage{bbold}

\usepackage{graphicx}

\usepackage[all]{xy}

\newcommand{\C}{\mathbb{C}}
\newcommand{\tr}[1]{\mathrm{tr}(#1)}
\newcommand{\xb}{X}

\newcommand{\wb}{W}

\newcommand{\aq}{/\!\!/}
\newcommand{\X}{\mathfrak{X}}

\newcommand{\R}{\mathfrak{R}}
\newcommand{\hm}{\mathrm{Hom}}

\newcommand{\F}{\mathtt{F}}

\newcommand{\xt}{\mathtt{x}}
\newcommand{\yt}{\mathtt{y}}
\newcommand{\wt}{\mathtt{w}}

\newcommand{\id}{I}

\newcommand{\ti}[1]{t_{(#1)}}

\newcommand{\SLm}[1]{\mathrm{SL}(#1,\C)}
\newcommand{\GLm}[1]{\mathrm{GL}(#1,\C)}
\newcommand{\glm}[1]{\mathfrak{gl}(#1,\C)}

\newcommand{\SUm}[1]{\mathrm{SU}(#1)}

\newcommand{\norm}[1]{|\!|#1 |\!|}
\newcommand{\End}{\mathrm{End}}

\makeatother

\begin{document}

\bibliographystyle{amsalpha} %amsplain

\title[Topology of moduli of free group representations]{The topology of moduli spaces of free group representations}

\author{Carlos Florentino and Sean Lawton}

\begin{abstract}
For any complex affine reductive group $G$ and a fixed choice of maximal compact subgroup $K$,
we show that the $G$-character variety of a free group strongly deformation retracts to the corresponding $K$-character space, which is a real semi-algebraic set.  Combining this with constructive invariant theory and classical topological methods, we show that the $\SLm{3}$-character variety of a rank 2 free group is homotopic to an 8 sphere and the $\SLm{2}$-character variety of a rank 3 free group is homotopic to a 6 sphere.
\end{abstract}

\maketitle

%\tableofcontents

\section{Introduction}
Let $G$ be a complex affine reductive group, and let $K$ be a maximal compact subgroup.  Define $\R_r(G)=\mathrm{Hom}(\F_r,G)$ to be the set of homomorphisms of a rank $r$ free group $\F_r$ into $G$.  The conjugation action of $G$ on $\R_r(G)$ is regular and the categorical quotient $\X_r(G)=\R_r(G)\aq G$ is a singular reduced algebraic set (irreducible if $G$ is irreducible).  There is a related space (a semi-algebraic set) $\X_r(K)=\mathrm{Hom}(\F_r,K)/K$.

In 2001 Bratholdt and Cooper \cite{BC} showed, among other things, that $\X_r(\SLm{2})$ strongly deformation retracts to $\X_r(\SUm{2})$.  The purpose of this paper is to generalize this result.  In \cite{BC} the authors exploit the fact that $\SLm{2}$ acts as an isometry group on hyperbolic 3-space to prove their theorems.  It is not clear how to directly generalize their methods to even $\SLm{3}$. However, our main theorem states that this holds in much greater generality: 

\begin{thm}
\label{thm:main}
Let $G$ be a complex affine reductive group, and let $K$ be a maximal compact subgroup.  Then $\X_r(G)$ strongly deformation retracts onto $\X_r(K)$.  In particular, they have the same homotopy type.
\end{thm}

Moreover, we classify the homeomorphism types in the cases $(r,n)=(1,n),(r,1),(2,2),(3,2)$ and $(2,3)$ for $K=\SUm{n}$.

The free groups $\F_r$ which we consider can be realized as the fundamental group of a surface with non-empty boundary.  The case of closed surface groups have been extensively studied due to the relationship between these character varieties and the moduli spaces of vector bundles and Higgs bundles.  In these cases, topological questions have almost exclusively been addressed using ``infinite dimensional Morse theory''.  Such methods do not immediately apply to the study of character varieties of free groups, and our main result completely reduces all such questions to the study of compact quotients.  

It is natural to speculate if the above theorem can be further generalized from free groups to arbitrary finitely generated groups (or at least to closed surface groups). Abelian representations, corresponding to representations of the fundamental group of a torus, are a non-counterexample.  It turns out that the ``twisted character varieties'' of closed surface groups (those which provide the correspondence to Higgs bundles and vector bundles) provide counterexamples for the cases of genus $g\geq2$. 

The paper is organized as follows.  In section 2 we introduce the basic definitions of character spaces and set down the notation for the remainder of the paper.  In section 3 we show there exists an auxiliary space that strongly deformation retracts to the $K$-character space (not generally an algebraic set) using classical Lie theory alone.  In section 4 we introduce the reader to the powerful ideas from Kempf-Ness and use them to show that the auxiliary space from section 3 also strongly deformation retracts to the $G$-character variety  .  This allows us to prove a weaker version of our main theorem; namely, that $\X_r(K)$ and $\X_r(G)$ are homotopy equivalent.  In section 5, we make some natural choices and show that we can conclude the stronger statement that the homotopy equivalence is a strong deformation retraction, thus finishing the proof of Theorem \ref{thm:main}.  The discussion becomes slightly technical, but we dispurse it into a series of lemmas to ease the reading.  We end section 5 with a brief discusion of how the deformation retraction relates to symplectic reduction.  In section 6 we compute the homeomorphism types of a number of examples of $\SUm{n}$ character varieties, and conjecture that the list provides a complete classification of the examples that are topological manifolds, and briefly say why we believe this to be true.  In the last section we discuss abelian representations and show our main theorem does extend to some finitely generated groups beyond free groups, and say more about the fact our results do not extend to all finitely generated groups.

\section*{Acknowledgments}

The first author thanks  J. Mour\~ao and J. P. Nunes, for several motivating conversations on the geometry of unitary surface group representations.  The second author thanks W. Goldman for introducing him to the general topic of character varieties, and to the Instituto Superior T\'ecnico for its support during the $2007-2008$ academic year.  We also thank Peter Gothen for helpful conversations, and the referee whose suggestions improved the current version of this paper. This work was partially supported by the Center for Mathematical Analysis, Geometry and Dynamical Systems at I.S.T., and by the ``Funda\c c\~ao para a Ci\^encia e a Tecnologia" through the programs Praxis XXI, POCI/MAT/58549/2004 and FEDER.  
Graphics herein created with Wolfram/Mathematica.

\section{Preliminaries}

\subsection{Complex affine reductive groups}

Throughout this paper $G$ will be a complex affine reductive group.  

An algebraic group is a group that is an algebraic variety (reduced zero set of a finite number of polynomials) such that the group operations are all regular (polynomial) mappings.  A complex affine group is an algebraic group that is the complex points of an affine variety.  Any affine group has a faithful linear representation (see \cite{Do} for instance), thus it is a closed subgroup of a general linear group and hence a linear Lie group.   Lie groups are smooth, and irreducible complex varieties are connected (see for instance \cite{Sh} page 321).

Let $K$ be a compact Lie group.  Then $K$ is a real algebraic group which embeds in $\mathrm{O}(n,\mathbb{R})$ for some $n$.  Since $K$ is algebraic there is an ideal $\mathfrak{I}$ in the real coordinate ring $\mathbb{R}[\mathrm{O}(n,\mathbb{R})]$ defining its points.  Let $G=K_{\C}$ be the complex zeros of $\mathfrak{I}$, called the \emph{complexification} of $K$.  Then $G$ is a complex affine subgroup of $\mathrm{O}(n,\C)$ with coordinate ring $\C[G]=\mathbb{R}[K]\otimes_{\mathbb{R}}\C$.  Any complex affine group $G$ which arises in this fashion is called {\it reductive}.  The ``unitary trick'' shows $\SLm{n}$ is reductive.  We note that this definition, although not the most general, coincides with all more general notions of reductivity when the algebraic group is complex linear.  In particular, another equivalent definition is that a complex linear algebraic group $G$ is reductive if for every finite dimensional representation of $G$ all subrepresentations have invariant complements.  The important observation is that such groups act like and have the algebraic structure of compact groups.  See \cite{Sch1}.

For example, $\mathrm{U}(n)=\{M\in\GLm{n}\ | \ M\overline{M}^{\mathsf{t}}=\id\}$, where $\id$ is the $n\times n$ identity matrix and $M^\mathsf{t}$ is the transpose of $M$.  Writing $M=A+\sqrt{-1}B$, we have that $\mathrm{U}(n)\cong$ $$\left\{\left(\begin{array}{cc}A&B\\-B&A\end{array}\right)\in \mathrm{GL}(2n,\mathbb{R})\ \ |\ A^\mathsf{t} A+B^\mathsf{t} B=I\ \&\ A^\mathsf{t} B- B^\mathsf{t} A=0\right\},$$  which sits isomorphically in $\mathrm{GL}(2n,\C)$ as $$\left\{\left(\begin{array}{cc}k&0\\0&(k^{-1})^\mathsf{t}\end{array}\right)\in \mathrm{GL}(2n,\C)\ |\ k\in \mathrm{U}(n)\right\}.$$  Letting $k$ be arbitrary in $\GLm{n}$ realizes the complexification $\mathrm{U}(n)_{\C}=\GLm{n}$.  In this way $\mathrm{U}(n)$ becomes the real locus of the complex variety $\GLm{n}$. Similarly, $\mathrm{SU}(n)_{\C}=\SLm{n}$.

\begin{rem}
We have not assumed $K$ is connected. Any compact Lie group $K$ has a finite number of connected components, 
all homeomorphic to the component containing the identity.   
As an algebraic variety, $\C[K_{\C}]$ has irreducible algebraic components (with respect to the Zariski topology).  
However, in this case the irreducible algebraic components are all disjoint homeomorphic topological components 
(with respect to the usual ball topology on $K_{\C}$), and each arises by complexifying a component of $K$ (see \cite{B} page 87).  
Thus we lose no generality in what follows by assuming that $G$ is an irreducible variety, although our arguments do not require it.
\end{rem}

\subsection{Representation Spaces}

Let $\F_r=\langle \xt_1,...,\xt_r\rangle$ be a rank $r$ free group (non-abelian), $K$ a compact Lie group and $G=K_{\C}$ its complexification.  Any homomorphism $\rho:\F_r\to G$ is a representation since $G\subset \mathrm{GL}(V)$ for some complex vector space $V$.  We call the set $\R_r(G)=\mathrm{Hom}(\F_r,G)$ the $G$-\emph{representation variety of} $\F_r$, and say $\rho \in \R_r(G)$ is $G$-{\it valued}.  The evaluation map, $$\rho\mapsto (\rho(\xt_1),...,\rho(\xt_r)),$$ gives a bijection between $\R_r(G)$ and $G^{\times r}$.  Since $G$ is a smooth affine variety, $\R_r(G)$ is likewise a smooth affine variety.  We note that we can assume without loss of generality that $\mathbb{R}[K]$ has no nilpotents.  Thus, $\C[K_{\C}]$ is reduced (no nilpotents) as well which implies $\R_r(G)$ is reduced.

The conjugation action of $G$ on $\R_r(G)$ is regular; that is, $G\times \R_r(G)\to \R_r(G)$ is regular (such mappings are rational functions in the matrix entries of $G$).  In particular, this action is either $(g,\rho)\mapsto g\rho g^{-1}$ or $$\left(g,\left(\rho(\xt_1),...,\rho(\xt_r)\right)\right)\mapsto \left(g\rho(\xt_1)g^{-1},...,g\rho(\xt_r)g^{-1}\right)$$ depending on whether we are working with $\mathrm{Hom}(\F_r,G)$ or $G^{\times r}$, respectively.  We often switch back and forth as is convenient.

We likewise have the real algebraic variety $\R_r(K)\cong K^{\times r}$, where $K$ is a maximal compact subgroup of $G$.

\subsection{Character Spaces}

\subsubsection{The complex quotient $\X_r(G)$}
The quotient $\R_r(G)/G$ is not generally a variety. One exception is when $G=\GLm{1}\cong
\C^*$ and $K=\mathrm{U}(1)\cong S^1$.  In these cases the conjugation action is trivial and thus $\R_r(\mathrm{U}(1))/\mathrm{U}(1)\cong (S^1)^{\times r}$ is the geometric $r$-torus.  Complexifying we find that $\R_r(\GLm{1})/\GLm{1}\cong (\C^*)^{\times r}$ is the algebraic $r$-torus.  Otherwise, the quotient is generally not Hausdorff.  For instance when $G=\SLm{2}$ and $r=1$ one cannot separate, using only continuous functions, the conjugation orbit of $\left(\begin{array}{cc}1&1\\0&1\end{array}\right)$ from the identity 
2 by 2 matrix.  
Therefore, we wish to approximate the space $\R_r(G)/G$ by the best space possible 
(it turns out requiring only for it to be Hausdorff gives an algebraic variety).

A theorem of Nagata \cite{Na} says that if a reductive group acts on a finitely generated algebra $A$, then the subalgebra of invariants $A^G=\{a\in A\ |\ g\cdot a=a\}$ is likewise finitely generated.  This is one answer to Hilbert's fourteeth problem.

Since $\R_r(G)$ is an affine variety, its coordinate ring $\C[\R_r(G)]$ is a finitely generated reduced ring (no nilpotents), and since $G$ acts on $\R_r(G)$ it acts on $\C[\R_r(G)]$ by $(g,f(\rho))\mapsto f(g^{-1}\rho g)$.  Thus $\C[\R_r(G)]^{G}$ is a finitely generated reduced ring, and consequently we define $$\X_r(G)=\mathrm{Spec}_{max}\left(\C[\R_r(G)]^{G}\right),$$ the set of maximal ideals, to be the $G$-\emph{character variety of} $\F_r$.  It is a singular affine variety.

It can be shown that $\X_r(G)$ is the categorical quotient $\R_r(G)\aq G$ in the category of affine varieties (or Hausdorff spaces or complex analytic varieties \cite{Lu2,Lu3}).  We recall the definition of a categorical quotient to be concrete.

\begin{defn}
A {\it categorical quotient} of a variety $V_G$ with an algebraic group $G$ acting rationally is an object $V_G\aq G$ and a $G$-invariant morphism $\pi_G:V_G\to V_G\aq G$ such that the following commutative diagram exists 
uniquely for all invariant morphisms $f:V_G\to Z$: $$ \xymatrix{
V_G \ar[rr]^-{\pi} \ar[dr]_{f} & & V_G\aq G \ar@{.>}[dl]\\
& Z &} $$
It is a {\it good} categorical quotient if the following additionally hold:
\begin{enumerate}
\item[$(i)$] for open subsets $U\subset V_G\aq G$, $\C[U]\cong \C[\pi^{-1}(U)]^{G}$
\item[$(ii)$] $\pi$ maps closed invariant sets to closed sets
\item[$(iii)$] $\pi$ separates closed invariant sets.
\end{enumerate}
\end{defn}

When $G$ is reductive and $V_G$ is an affine $G$-variety, then $$V_G\to \mathrm{Spec}_{max}(\C[V_G]^G)$$ is a good categorical quotient.  See \cite{Do} for details.

The representations that are {\it stable} (points with closed orbits having zero dimensional isotropy) form a Zariski open set (dense if non-empty), denoted $\R_r^{s}(G)$.  The universal (surjective) morphism $\pi_{G}:\R_r(G)\to \X_r(G)$ restricts to a geometric quotient $\R_r^{s}(G)\to\R_r^{s}(G)/G\subset \X_r(G)$.

For instance, when $G=\SLm{n}$ these are exactly the irreducible representations (having no proper invariant subspace with respect to the action on $\C^n$).

The poly-stable representations, denoted by $\R_r^{ps}(G)$, are defined to be those representations with closed $G$-orbits. For instance, when $G=\SLm{n}$ these are exactly the completely reducible representations (decomposable into a direct sum of irreducible sub-representations with respect to the action on $\C^n$).

Any representation can be continuously and conjugate-invariantly deformed to one that is poly-stable.  Hence the points of $\X_r(G)$ are unions of orbits of representations that are deformable in this way.  Such a union, denoted hereafter by $[[\rho]]_G$ is called an {\it extended orbit equivalence class}.  Thus $\X_r(G)$ parametrizes extended orbit equivalences of representations.  $\R_r^{ps}(G)/G$ is in one-to-one correspondence with the space $\X_r(G)$ since every extended orbit $[[\rho]]_G\in \X_r(G)$ is a union of $G$-orbits, denoted by $[\rho]_G$, and the one of smallest dimension must be closed.  Two $G$-orbits in $\R_r(G)/G$ are identified in $\X_r(G)$ if and only if their closures intersect, and all orbits of elements in $\R_r^{s}(G)$ are homeomorphic.  

%\begin{rem}
%Recently, A. Sikora \cite{Si4} has introduced a notion of irreducible and completely reducible for $G$-representations for arbitrary complex reductive groups $G$.  A representation is irreducible if its image is not contained in any proper parabolic subgroup of $G$ and is completely reducible if for all parabolics that contain its image there is also a Levi subgroup containing it.  In these terms, he shows that poly-stable represententations are completely reducible and stable representations are irreducible (with respect to the action of $G$ modulo its center).
%\end{rem}

Any such reductive quotient has an affine lift (see \cite{MFK}).  In otherwords, there is an affine space $\mathbb{A}^N$ for some potentially large $N$ where $\R_r(G)\subset \mathbb{A}^N$ and where the action of $G$ extends.  Then $$\Pi:\C[\mathbb{A}^N]\longrightarrow \C[\R_r(G)]$$ and more importantly
$$\Pi_{G}:\C[\mathbb{A}^N\aq G]\longrightarrow \C[\R_r(G)\aq G]$$ are surjective morphisms.  

In fact, since $G$ is linear and hence given by $n\times n$ matrices for some $n$, we may take $\mathbb{A}^N=\mathfrak{gl}(n,\C)^{\times r}$ and the action of $G$ to be, as it is on $G^{\times r}$, diagonal conjugation.

The coordinate ring of this affine space is $$\C[\mathfrak{gl}(n,\C)^{\times r}]=\C[x^{k}_{ij}\ |\ 1\leq i,j \leq n, \ 1\leq k\leq r],$$ the complex polynomial ring in 
$rn^2$ variables.

Let $$\xb_k=\left(
\begin{array}{cccc}
x^k_{11} & x^k_{12} & \cdots & x^k_{1n}\\
x^k_{21} & x^k_{22}  &\cdots & x^k_{2n}\\
\vdots &\vdots & \ddots & \vdots\\
x^k_{n1} & x^k_{n2} &\cdots &x^k_{nn}\\
\end{array}\right)$$  be a \emph{generic matrix} of size $n\times n$.
In 1976 (see \cite{P1}) Procesi proves (in the above terms)

\begin{thm}[$1^{\text{st}}$ Fundamental Theorem of Invariants of $n\times n$ Matrices]\label{prfd1}
$$\C[\glm{n}^{\times r}\aq \SLm{n}]=\C[\tr{\xb_{i_1}\xb_{i_2}\cdots \xb_{i_l}}\ |\ 1\leq l \leq d(n)],$$ where $d(n)$ is a fixed positive integer dependent only on $n$. 
\end{thm}

The number $d(n)$ is called the \emph{degree of nilpotency}.  The only values known are $d(2)=3$, $d(3)=6$ and $d(4)=10$.  

Therefore, $\C[\R_r(\SLm{n}\aq \SLm{n}]$ is generated by $$\Pi(\tr{\xb_{i_1}\xb_{i_2}\cdots \xb_{i_l}})=\tr{\widehat{\xb_{i_1}}\widehat{\xb_{i_2}}\cdots \widehat{\xb_{i_l}}}$$ where $\widehat{\xb_k}=(\widehat{x^k_{ij}})$ and $\widehat{x_{ij}^k}=\Pi(x^k_{ij})\in \C[\R_r(\SLm{n}].$  We will return to this case latter in the paper.

\begin{rem}
$\X_r(G)$ is not necessarily, what we call the Culler-Shalen Character Variety, here denoted $\mathcal{CV}_r(G)$ $($see \cite{CS}$)$.  By definition $\mathcal{CV}$ is the space which corresponds to the algebra generated by traces of representations.  \cite{P1} essentially shows $\mathcal{CV}_r(G)=\X_r(G)$ for $G=\SLm{n}$or $\GLm{n}$.  A recent preprint of Sikora \cite{Si3} notes that this can be generalized to any finitely generated group $\Gamma_g$ and additionally $G$ equal to 
$\mathrm{O}(n,\C)$ or $\mathrm{Sp}(n,\C)$; that is, the coordinate ring of $\C[\mathrm{Hom}(\Gamma_g,G)\aq G]$ is generated by traces of representations in these cases.  However, this fact is not always true.  For $G=\mathrm{SO}(n,\C)$, the coordinate rings may be different from the trace algebras, and hence the quotients will differ as well.  For instance $\mathrm{SO}(2,\C)$-representations of a rank 2 free provide a concrete instance of this phenomenon.
\end{rem}

\begin{rem}
Let $G_\mathbb{R}$ be the real points of the complex variety $G$.  Although the categorical quotient mapping $\pi_G:V_G\to V_G\aq G$ is always surjective, the corresponding mapping $\pi_{G_{\mathbb{R}}}$ need not be surjective.  For instance, the quotient mapping $\SUm{2}\to \mathrm{SU}(2)/\SUm{2}$ (for the conjugation action) is a mapping onto the interval $[-2,2]$, but the corresponding space $\SUm{2}\aq \SUm{2}=\mathbb{R}$, and so the projection mapping $\SUm{2}\to \mathrm{SU}(2)\aq\SUm{2}$ is very far from surjective.
\end{rem}

\subsubsection{The real quotient $\X_r(K)$}

Let $K$ be a compact Lie group.  There is a related space $\X_r(K)=\mathrm{Hom}(\F_r,K)/K$, called the $K$-\emph{character space of} $\F_r$.  This space is always Hausdorff since all orbits of compact groups are closed.  It is a compact space (path-connected if $K$ is path-connected) since these properties are inherited by compact quotients and Cartesian products and the space $\R_r(K)\cong K^{\times r}$.

More generally, let $S$ be a real affine algebraic $K$-variety.  Then there is an equivariant closed embedding $S\hookrightarrow W$, where $W$ is a real representation of $K$.  Let $\mathbb{R}[W]^K$ be the ring of $K$-invariants in the real coordinate ring $\mathbb{R}[W]$.  The invariant ring is finitely generated by polynomials $p_1,...,p_d$ and the corresponding mapping $$p=(p_1,...,p_d):S\to \mathbb{R}^d$$ is proper and induces a homeomorphism $S/K\cong p(S)$.  This image is generally a semi-algebraic set; that is, a finite union of sets each determined by a finite number of polynomial inequalities.  There is an ideal $I$ so that $\mathbb{R}[S]^K\cong \mathbb{R}[p_1,...,p_d]/I$.  Let $Z_{\mathbb{R}}(I)$ be the real zeros of the generators of $I$ as a subset of $\mathbb{R}^d$.  Then one can show that $S/K$ is a closed semi-algebraic subset of $Z_{\mathbb{R}}(I)$.  See \cite{Sch1} and \cite{PS}.

Since $K\subset G$ implies $\R_r(K)\subset \R_r(G)$ and reductivity implies $\C[\R_r(G)]^G=\C[\R_r(G)]^K$, it follows that the real and imaginary parts of a set of generators for $\C[\R_r(G)]^G$ will give a set of generators for $\mathbb{R}[\R_r(K)]^K$.  For instance, $\mathbb{R}[\R_r(\SUm{n})]^{\SUm{n}}$ is generated by the real and imaginary parts of trace functions of the form $\tr{\widehat{\xb_{i_1}}\widehat{\xb_{i_2}}\cdots \widehat{\xb_{i_l}}}$.  Thus $\R_r(K)/K$ is generally a semi-algebraic set that is given as the real points of the variety $\R_r(G)\aq G$ subject to a finite number of polynomial inequalities.

\subsubsection{A brief word about topologies}\label{topologies}

We can topologize real semialgebraic sets as subspaces of Euclidean space (metric topology), just as we can topologize complex affine varieties as subspaces of complex affine space (again called the metric topology).  These real sets are topologically well behaved. For instance, the local structure around a nonisolated point is always isomorphic to a cone (see \cite{BCR} 9.3.6).  In particular, they are locally contractible (just like complex affine varieties).  We also know that compact semialgebraic sets may be triangulated (see \cite{BCR} 9.4.1), and so they are finite simplicial complexes (again complex affine varieties are triangulable as well).

Both $G^{\times r}$ and $K^{\times r}$ can be given the product topology
where $K\subset G$ is given the subspace topology.  The product topology is equivalent
to the compact-open topology of the representation spaces $\hm(\F_{r},G)\cong G^{\times r}$
and $\hm(\F_{r},K)\cong K^{\times r}$; in other words, the evaluation
map not only gives a bijection but a homeomorphism as well. This is easy to see since $\F_{r}$ is given the discrete topology and $G$ is a metric space and so the compact-open topology is equivalent to the
compact-convergence topology with respect to evaluations of words
in $\F_{r}$. This shows the compact-open topology is the product
topology since limits are determined only by convergence on generators.  

For a topological space $M$ with a $K$ action (a $K$-space), the orbit space $M/K$ may be given the usual quotient topology. $K$ not only acts on $K^{\times r}$ by simultaneous conjugation
in each factor, but also on $G^{\times r}$ since $G$ contains $K$.  In this way both $\R_r(K)/K$ and $\R_r(G)/K$ can have the usual $K$ quotient topology which is compatible with the metric topology coming from the semi-algebraic structures. Therefore, all such topologies are equivalent to the metric topologies.

Moreover, since $G^{\times r}$ is a variety, these other topologies contain the Zariski topology as a much coarser sub-topology.  In fact they contain the less coarse \'etale topology, and the even finer strong (complex analytic) topology (see \cite{Mu}).  Note we will generally use the symbol ``$\cong$'' to mean isomorphism in the category of topological spaces, and where we have considered more than one compatible topology we use the strongest one.

\section{The Equivariant Retraction $G\to K$}

Let $K$ be a compact Lie group and $G=K_{\C}$.  Denote the Lie algebras of $G$ and $K$ by $\mathfrak{g}$ and $\mathfrak{k}$, respectively. Then, letting $\mathfrak{g}=\mathfrak{k}\oplus\mathfrak{p}$
be a Cartan decomposition, the multiplication map\[
m:K\times\exp\mathfrak{p}\to G\]
 is a surjective diffeomorphism (see \cite{K}, page 384). The main example
for our purposes will be the case when $G=\SLm{n}$ and $K=\SUm{n}$,
for which the diffeomorphism above is the usual polar decomposition.
In general, the inverse to $m$ can be defined explicitly by\begin{eqnarray}
m^{-1}:G & \to & K\times\exp\mathfrak{p}\nonumber \\
g & \mapsto & \left(g(g^{*}g)^{-\frac{1}{2}},(g^{*}g)^{\frac{1}{2}}\right),\label{eq:m-1}\end{eqnarray}
where $g^{*}$ denotes the Cartan involution applied to $g$ (which, in the case of $\SLm{n}$ and $\SUm{n}$, is the usual conjugate
transpose map). The formula above follows from the fact that, if we
write $g=ke^{p}$, for $k\in K$ and $p\in\mathfrak{p}$, then $g^{*}=e^{p^{*}}k^{*}=e^{p}k^{*}$,
which implies $g^{*}g=e^{2p}$ (note that the Cartan involution fixes
any element of $\mathfrak{p}$).

Now, for any $t\in[0,1]\subset\mathbb{R}$, consider the map $\phi_{t}:G \to  G$ defined by
$$
g=ke^{p} \mapsto  g(g^{*}g)^{-\frac{t}{2}}=ke^{p}(e^{2p})^{-\frac{t}{2}}=ke^{(1-t)p}.
$$
It is clear that $\phi_t$ are continuous maps.  Let us denote the conjugation action of $K$ on $G$ by\[
k\cdot g=kgk^{-1},\quad\mbox{for }k\in K,g\in G.\]
 Then, we have

\begin{prop}\label{GKdeformation}
The collection $\left\{ \phi_{t}\right\} _{t\in[0,1]}$ above defines
a strong deformation retraction from $G$ to $K$, given by $r(g)=\phi_{1}(g)=g(g^{*}g)^{-\frac{1}{2}}$.
Moreover, for any $t\in[0,1]$, $\phi_{t}$ is $K$-equivariant in
the sense that $\phi_{t}(k\cdot g)=k\cdot\phi_{t}(g)$. 
\end{prop}
\begin{proof}
Clearly $\phi_{0}(g)=g$ and $\phi_{1}(g)\in K$ for all $g\in G$,
by (\ref{eq:m-1}). Also, we have $\phi_{t}(k)=k$ for all $k\in K$,
because in this case $k^{*}k$ is the identity. So, $\phi$ is indeed
a strong deformation retraction. To prove $K$-equivariance we first
note that, for any real number $t\geq0$, we have\begin{equation}
he^{tp}h^{-1}=\left(he^{p}h^{-1}\right)^{t}\label{eq:exp}\end{equation}
 for all $h\in K$ and $p\in\mathfrak{p}$. The formula certainly
works for rational $t$, and the general case follows by continuity.
Using this, we compute \begin{eqnarray*}
\phi_{t}(h\cdot g) & = & \phi_{t}(hgh^{-1})\\
 & = & hgh^{-1}\left((hgh^{-1})^{*}hgh^{-1}\right)^{-\frac{t}{2}}\\
 & = & hgh^{-1}\left(hg^{*}gh^{-1}\right)^{-\frac{t}{2}}\\
 & = & hgh^{-1}\left(he^{2p}h^{-1}\right)^{-\frac{t}{2}}\\
 & = & hgh^{-1}he^{-tp}h^{-1}\\
 & = & h\phi_{t}(g)h^{-1},\end{eqnarray*}
 for all $h\in K$ and $g=ke^{p}\in G$. 
\end{proof}

Now suppose that $T$ is any topological space, $S\subset T$ is a subspace, and a group $K$ acts on $T$ so the restriction to $S$ is stable (for all $k\in K$ and $s\in S$, $k\cdot s\in S$).  Then under these assumptions, it is elementary to prove

\begin{prop}\label{Kequivariantdeformation}
If there exists a $K$-equivariant (strong) deformation retraction of $T$ onto $S$, then for all integers $r>0$ there exists a (strong) deformation retraction of $T^{\times r}/K$ onto $S^{\times r}/K$.
\end{prop}

Putting these propositions together establishes

\begin{cor}\label{defretract}
$\X_r(K)$ is a strong deformation retraction of $\R_r(G)/K$.
\end{cor}

\begin{proof}
We first recall that $\X_r(K)\cong K^{\times r}/K$ and $\R_r(G)/K\cong G^{\times r}/K$.  Let $\phi_t$ be the deformation of $G$ to $K$ from Propositions \ref{GKdeformation}.   Then with respect to this mapping and the diagonal conjugation action, Proposition \ref{Kequivariantdeformation} is applicable, and so establishes the corollary.
\end{proof}

\section{Kempf-Ness Sets and Homotopy Equivalence of Character Spaces}

In this section we discuss results of \cite{KN}, \cite{Sch1}, and \cite{Ne} that are relevant to our situation.  

Let $G$ be complex affine reductive group, $V_G$ be an affine $G$-variety, $V_G\aq G=\mathrm{Spec}_{max}\left(\C[V_G]^G\right)$.  For us $V_G=\R_r(G)$ and the action is by diagonal conjugation.  However, we proceed more generally.

We may assume $V_G$ is equivariantly embedded as a closed subvariety of a representation $G\to \mathrm{GL}(V)$.  Let $\langle\ ,\ \rangle$ be a $K$-invariant Hermitian form on $V$ with norm denoted by $\norm{\ }$.

Define for any $v\in V$ the mapping $p_v:G\to \mathbb{R}$ by $g\mapsto \norm{g\cdot v}^2$.  It is shown in \cite{KN} that any critical point of $p_v$ is a point where $p_v$ attains its minimum value.  Moreover, the orbit $G\cdot v$ is closed and $v \not= 0$ if and only if $p_v$ attains a minimum value.  

Define $\mathcal{KN}\subset V_G\subset V$ to be the set of critical points $\{v\in V_G\subset V\ |\ (dp_v)_{\id}=0\}$, where $\id\in G$ is the identity.  This set is called the \emph{Kempf-Ness} set of $V_G$. 
Since the Hermitian norm is $K$-invariant, for any point in $\mathcal{KN}$, its entire $K$-orbit is also contained in $\mathcal{KN}$.
The following theorem is proved in \cite{Sch1} making reference to \cite{Ne}.

\begin{thm}[Schwarz-Neeman]\label{schwarzneeeman}
The composition $\mathcal{KN}\to V_G\to V_G\aq G$ is proper and induces a homeomorphism $\mathcal{KN}/K\to V_G\aq G$ where $V_G\aq G$ has the subspace topology induced from its equivariant affine embedding.  Moreover, there is a $K$-equivariant deformation retraction of $V_G$ to $\mathcal{KN}.$
\end{thm}

\begin{rem}
Let $V_G$ be an affine $G$-variety and $W_K\subset V_G$ a $K$-space. Then Theorem \ref{schwarzneeeman} implies whenever there exists a $K$-equivariant deformation retraction of $V_G$ onto $W_K$, the spaces $V_G\aq G$ and $W_K/K$ are homotopically equivalent.
\end{rem}

Recall our notation: $\X_r(G)$ is the $G$-character variety of $\F_r$, and $\X_r(K)$ is the $K$-character space of $\F_r$, where $K$ is a maximal compact subgroup of $G$.  We now can prove one of the main results of this paper.

\begin{thm}
\label{homotopy-type}
Let $G$ be a complex affine reductive group, and let $K$ be a maximal compact subgroup.  
Then $\X_r(K)$ and $\X_r(G)$ have the same homotopy type.
\end{thm}

\begin{proof}
From Corollary \ref{defretract} we have that $\X_r(K)$ and $\R_r(G)/K$ have the same homotopy type.  From Theorem \ref{schwarzneeeman} we have that $\R_r(G)/K$ and $\mathcal{KN}/K\cong \X_r(G)$ have the same homotopy type.  Since homotopy equivalence is an equivalence relation, and so transitive, we conclude $\X_r(K)$ and $\X_r(G)$ have the same homotopy type.  In fact, two spaces are homotopy equivalent if and only if both spaces deformation retract to a third auxiliary space (see \cite{Hatcher}).
\end{proof}

\begin{cor}
The homotopy groups, homology groups, and cohomology groups of $\X_r(K)$ and $\X_r(G)$ are all isomorphic.
\end{cor}

As a particular application of these isomorphisms, we mention the case of $G=\SLm{2}$.

\begin{cor}
The Poincar\'e polynomial for $\X_r(\SLm{2})$ is 
$$P_t(\X_r)=1+t-\frac{t(1+t^3)^r}{1-t^4}+\frac{t^3}{2}\left(\frac{(1+t)^r}{1-t^2}-\frac{(1-t)^r}{1+t^2}\right) .$$  
\end{cor}
\begin{proof}
In 2008 T. Baird \cite{Ba}, using methods of equivariant cohomology, showed that the 
Poincar\'e polynomial for $\X_r(\SUm{2})\cong\SUm{2}^{\times r}/\SUm{2}$ is $P_t$ above.
Thus, Theorem \ref{homotopy-type} establishes the corollary.
\end{proof}

\section{Strong Deformation Retraction of Character Varieties}

In general, the Kempf-Ness set depends on a choice of $G$-equivariant embedding $\R_r(G)\hookrightarrow V\cong \C^N$ and a choice of $K$-invariant Hermitian form on $V$, for a fixed choice of maximal compact $K$, where $G=K_{\C}$.  We show that making natural choices allows for a strong deformation retraction of $\X_r(G)=\R_r(G)\aq G$ onto $\X_r(K)=\R_r(K)/K$.  

However, we first show that the topological $K$-quotient $\X_r(K)$ is a subspace of the categorical $G$-quotient $\X_r(G)$.

The following lemma, which seems to be standard despite our lack of reference, will prove useful.  See the appendix for a proof.
\begin{lem}\label{unitarylemma}
If two $K$-valued free group representations are conjugate by an element in $G$, then they are conjugate by an element in $K$. 
\end{lem}

\begin{prop}
There exits an injection $\iota:\X_r(K)\hookrightarrow \X_r(G)$. 
\end{prop}
\begin{proof}
Let $[\rho]_G$ be the $G$-orbit of the representation
$\rho$. Then $\iota$ is defined as follows: $[\rho]_K\mapsto [[\rho]]_G$,
where $[[\rho]]_G$ is the union of all $[\rho]_G$
such that $\overline{[\rho]_{G}}\cap\overline{[\rho^{\prime}]_G}\not=\emptyset$.
This is clearly well defined since $K\subset G$.

Let $\R_r^{ps}(G)$ be the set consisting of all representations with closed orbits (poly-stable points).  Since $K$ is compact, all representations in $\R_r(K)$ have closed $K$-orbits, which implies they have closed $G$-orbits as well. Consequently, Lemma \ref{unitarylemma} shows there is an injection $\R_r(K)/K\hookrightarrow\R_r^{ps}(G)/G$, given by $[\rho]_K\mapsto[\rho]_G$.

Since $\R_r^{ps}(G)/G$ is in one-to-one correspondence with $\X_r(G)$, each
orbit $[[\rho]]_G$ has a poly-stable representative
$\rho^{(ps)}\in[[\rho]]_G$ whose $G$-orbit
is closed. The bijection $\R_r^{ps}(G)/G\to\X_r(G)$ is given by $[\rho]_G\mapsto[[\rho]]_G$.
Therefore, $\iota$ is injective since it is the composit of two injections:
$[\rho]_K\mapsto[\rho]_G\mapsto[[\rho]]_G.$\end{proof}

\begin{rem}
An algebraic proof: the real and imaginary parts of a set of generators for $\C[\R_r(G)]^G$ will generate $\mathbb{R}[\R_r(K)]^K$.  Let these generators be denoted $p=(p_1,...,p_{2N})$.  Then $\R_r(K)/K\cong p(\R_r(K))\subset p(\R_r(G))\cong \R_r(G)\aq G$.  In these terms, $\iota$ is simply the inclusion mapping; in particular, $\iota$ is injective.
\end{rem}

\begin{prop}
There exists a continuous surjection $$\pi_{G/K}:\R_r(G)/K\to \X_r(G)$$
such that $\pi_{G/K}(\X_r(K))=\iota(\X_r(K))$. 
\end{prop}
\begin{proof}
This mapping is given by extending the definition of $\iota$ above
to all of $\R_r(G)/K$. Namely, $[\rho]_K\mapsto[[\rho]]_G$.
Since $K\subset G$, $\pi_{G/K}$ is well defined. For every $\rho\in\R_r(G)$
there exists orbits $[\rho]_K\subset[[\rho]]_G.$
Therefore, $\pi_{G/K}$ is surjective. By definition, it is equal
to $\iota$ when restricting to $\X_r(K)$ and hence $\pi_{G/K}(\X_r(K))=\iota(\X_r(K))$.
For continuity, consider the commutative diagram:

$$
\xymatrix{ & \R_r(G)\ar[dl]_{\pi_{K}}\ar[dr]^{\pi_{G}}\\
\R_r(G)/K\ar[rr]_{\pi_{G/K}} &  & \X_r(G)}
$$

Take $U$ open in $\X_r(G)$.  The morphism $\pi_{G}$ is given
by polynomials and so is continuous in the strongest of the compatible topologies discussed in Section \ref{topologies} . Consequently,
$V=\pi_{G}^{-1}(U)$ is open and $G$-invariant and therefore $K$-invariant
since $K\subset G$. Hence, $\pi_{K}(V)$ is also open by definition
of the quotient topology. By commutativity of the above diagram $\pi_{G/K}$
is then continuous.
\end{proof}
\begin{cor}
The injection $\iota$ is continuous. 
\end{cor}
\begin{proof}
In the proof of the last proposition we saw that $\pi_{G/K}\big|_{\X_r(K)}=\iota$,
in other words the following diagram is commutative:

$$
\xymatrix{ & \X_r(K)\ar[dl]_{\mathrm{id}}\ar[dr]^{\iota}\\
\R_r(G)/K\ar[rr]_{\pi_{G/K}} &  & \X_r(G)}
$$

The continuity of $\iota$ then follows from the continuity of $\pi_{G/K}$.
\end{proof}

\begin{prop}
\label{section} There exists a continuous mapping $$\sigma:\X_r(G)\to \R_r(G)/K$$
such that $\pi_{G/K}\circ\sigma=\mathrm{id}$, and $\sigma\big|_{\iota(\X_r(K)}=\mathrm{id}\big|_{\X_r(K)}$. 
\end{prop}
\begin{proof}
The Kempf-Ness set provides a continuous mapping $\sigma$ by composing the homeomorphism with the inclusion in the following commutative diagram:

$$\xymatrix{
\R_r(G)/K \ar[dr]^{\pi_{G/K}} & \R_r(G)\ar[l]_{\pi_K}\ar[d]^{\pi_G}\\
\mathcal{KN}/K \ar[u]^{\cup} &  \X_r(G)\ar[l]_{\cong}
}.$$

Letting the homeomorphism $\mathcal{KN}/K\cong \X_r(G)$ be denoted by $h$, we note that $h$ picks out a representative from $[[\rho]]_G$ and $\pi_{G/K}$ maps any element $[\rho]_K$ to $[[\rho]]_G$.  Therefore, $\pi_{G/K}\left(h([[\rho]]_G)\right)=[[\rho]]_G$.

We now argue that $\sigma$ is ``$K$-preserving.''  Since $G$ is linear, there exists $n>0$ so $G$ acts on $V=\mathfrak{gl}(n,\C)^{\times r}\cong\C^{rn^2}$ by simultaneous conjugation.  Thus, $\R_r(G)$ naturally sits inside $V$ equivariantly.  For generic matrices $x,y$ we can define $\langle x,y\rangle=\tr{xy^{*}}$, where $y^*$ is the Cartan involution.  For any element $k\in K$ we have $\langle kxk^{-1},kyk^{-1}\rangle=\langle x,y\rangle$.  Then for $2r$ generic matrices $x_i,y_i$ $$\langle (x_1,...,x_r) ,(y_1,...,y_r)\rangle=\sum_{i=1}^r \langle x_i,y_i\rangle$$  defines a $K$-invariant Hermitian form.

To prove the proposition we must show that whenever $\rho$ is $K$-valued the Kempf-Ness set maps $[[\rho]]_G$ to a $K$-conjugate of $\rho$.

The Kempf-Ness set is defined by condition $(dp_v)_{I}=0$ where $p_v(g)=\norm{g\cdot v}^2=\langle g\cdot v,g\cdot v \rangle$.  Let $g_t$ be a sequence of elements in $G$ so $g_0=\id$ and let $v=\rho$ be a representation. Then 
\begin{align*}
\frac{d}{dt}p_v(g_t)&=\frac{d}{dt}\langle g_t\cdot v,g_t\cdot v \rangle\\
&=\frac{d}{dt}\langle g_t\rho g_t^{-1},g_t\rho g_t^{-1} \rangle\\
&=\langle \frac{d}{dt}(g_t\rho g_t^{-1}),\rho \rangle +\langle \rho, \frac{d}{dt}(g_t\rho g_t^{-1})\rangle.
\end{align*}
We now compute $$\frac{d}{dt}(g_t\rho g_t^{-1})_0=\left(\frac{d g_t}{dt}\right)_0\rho g^{-1}_0-g_0\rho g^{-1}_0\left(\frac{d g_t}{dt}\right)_0 g_0^{-1}=\left(\frac{d g_t}{dt}\right)_0\rho-\rho\left(\frac{d g_t}{dt}\right)_0.$$  We may identify any tangent space with the tangent space at the identity by right invariant vector fields.  Right translate by $\rho^{-1}$ and define $u_0=-\left(\frac{d g_t}{dt}\right)_0$, then define $$u(\wt)=\frac{d}{dt}(g_t\rho g_t^{-1})_0\big|_{\rho^{-1}}=\mathrm{Ad}_\rho(u_0)-u_0,$$ for any $\wt\in \F_r$.

Putting these computations together we have
$$\frac{d}{dt}p_v(g_t)_0=\langle \mathrm{Ad}_\rho(u_0)-u_0,\rho \rangle +\langle \rho, \mathrm{Ad}_\rho(u_0)-u_0\rangle.$$

When $\rho$ is $K$-valued we have
\begin{align*}
\langle \mathrm{Ad}_\rho(u_0)-u_0,\rho \rangle&=\langle \rho(u_0)\rho^{-1},\rho  \rangle-\langle u_0,\rho \rangle\\
&=\langle \rho(u_0)\rho^{-1},\rho  \rangle-\langle \rho (u_0)\rho^{-1},\rho \rho \rho^{-1} \rangle\\
&=0,
\end{align*}
since $\langle\ ,\ \rangle$ is $K$-invariant (Hermitian).  Likewise, $\langle \rho, \mathrm{Ad}_\rho(u_0)-u_0 \rangle=0$. 

Therefore, whenever $[[\rho]]_G$ contains a $K$-valued representation it is included in the Kempf-Ness set.  We then conclude ``$\sigma$ is $K$-preserving:''  $\sigma\circ \iota(\X_r(K))\subset \X_r(K)$.   This implies $\iota\circ \sigma \circ \iota= \pi_{G/K}\circ \sigma \circ \iota=\iota\circ \mathrm{id}.$
Then since $\iota$ is injective, we conclude $\sigma\circ \iota=\mathrm{id}$ on $\X_r(K)$.\end{proof}

\begin{thm}
There exists a family of mappings $\Phi_{t}^{\sigma}:\X_r(G)\to\X_r(G)$ that determine a strong
deformation retraction of $\X_r(G)$ onto $\X_r(K)$. 
\end{thm}
\begin{proof}
Our previous propositions establish that the following diagram is
commutative: \[
\xymatrix{\R_r(G)/K\ar[dd]^{\pi_{G/K}}\ar[rr]^{\Phi_{t}} &  & \R_r(G)/K\ar[dd]^{\pi_{G/K}}\\
& \X_r(K)\ar@{^{(}->}[ul]_{\mathrm{id}}\ar@{^{(}->}[ur]^{\mathrm{id}}\ar@{^{(}->}[dl]^{\iota}\ar@{^{(}->}[dr]_{\iota}\\
\X_r(G)\ar@{-->}@/^3pc/[uu]^{\sigma} &  & \X_r(G)}
\]

Define $\Phi_{t}^{\sigma}=\pi_{G/K}\circ\Phi_{t}\circ\sigma$. Then
since all composit maps are continuous, so is $\Phi_{t}^{\sigma}$.
We now verify the other properties of a strong deformation retraction.
Firstly, $\Phi_{0}^{\sigma}$ is the identity since $\Phi_{0}=\mathrm{id}$
and $\pi_{G/K}\circ\sigma=\mathrm{id}$.

Next, we show $\Phi_{1}^{\sigma}$ is into $\iota(\X_r(K))$.
Since $\Phi_{1}(\R_r(G)/K)\subset \X_r(K)$, it follows
$W=\Phi_{1}\left(\sigma(\X_r(G))\right)\subset \X_r(K)$.
Moreover, $\pi_{G/K}=\iota$ on $\X_r(K)$, so $\pi_{G/K}(W)=\iota(W)\subset\iota(\X_r(K))$.

Lastly, we show that for all $t$, $\Phi_{t}^{\sigma}$ is the identity
on $\iota(\X_r(K))$. In deed, we have shown that $(\sigma\circ\iota)=\mathrm{id}$
on $\X_r(K)\subset \R_r(G)/K$. For all $t$, $\Phi_{t}$
is the identity on $\X_r(K)$. Again, using the fact that $\pi_{G/K}=\iota$
on $\X_r(K)$, we have for any point $[\psi]\in \X_r(K)$,
\[
\iota([\psi])\mapsto\sigma(\iota([\psi]))=[\psi]\mapsto\Phi_{t}([\psi])=[\psi]\mapsto\pi_{G/K}([\psi])=\iota([\psi]),\]
 as was to be shown. 
\end{proof}

\subsection{A brief word about symplectic reduction}
As before, let $K$ be a compact Lie group and $G=K_{\C}$ its complexification.  
The polar mapping $K\times \mathfrak{k} \to G$ is a $K$-biinvariant diffeomorphism.  
Let $TK$ be the tangent bundle of $K$.  
The mapping $TK\to K\times \mathfrak{k} \to G$ of the inverse $TK \to K\times \mathfrak{k}$ of left translation with the polar map is a diffeomorphism.  Consequently, $\R_r(G)\cong G^{\times r}\cong (TK)^{\times r}\cong T(\R_r(K))$ is the tangent bundle over $\R_r(K)\cong K^{\times r}$.  Therefore $\R_r(G)$ is a symplectic manifold by identifying $T(\R_r(K))$ with $T^*(\R_r(K))$, and since it also is a complex algebraic set it is therefore K\"ahler.

The action of $K_{\C}$ on $\R_r(G)$ by conjugation is holomorphic and the restriction to the action of $K$ is Hamiltonian.  According to \cite{Hu1} the momentum mapping $\mu: G=K\times \mathfrak{k}\to \mathfrak{k}\cong \mathfrak{k}^*$ is the mapping $\mu(x,Y)=\mathrm{Ad}_xY-Y$.  Consequently, since the action on $\R_r(G)\cong G^{\times r}$ is diagonal the momentum mapping $\mu: \R_r(G)\cong(K\times \mathfrak{k})^{\times r}\to \mathfrak{k}\cong \mathfrak{k}^*$ is the mapping $$\mu((x_1,...,x_r),(Y_1,...,Y_r))=\sum_{i=1}^r\mathrm{Ad}_{x_i}Y_i-Y_i.$$  Then by Proposition 4.2 of \cite{Hu2} the reduced space $\mu^{-1}(0)/K\cong \X_r(G)$ inherits a \emph{stratified K\"ahler structure}.

The real stratified symplectic structure on $\X_r(G)$ comes from the generalization by Lerman and Sjamaar \cite{LS} of the Marsden-Weinstein symplectic reduction \cite{MW} to a stratified structure when 0 is not a regular value of the momentum mapping.  Kirwan \cite{Ki} then showed using the Kempf-Ness theory already discussed that the null cone $\mu^{-1}(0)$ can be taken to be a Kempf-Ness set.  We mention this because from this point of view it is easier to see that the $K$-representations are a subset of the Kempf-Ness set, since $\R_r(K)$ sits, in this context, as the image of the zero section of the tangent bundle $\R_r(G)\cong T(\R_r(K))\to \R_r(K)$.  Then the momentum mapping is easily seen to map all such points to 0 (the values of all the $Y_i$'s are 0!).

It is also interesting to note that the complex dimension of $\X_r(G)$ can be generally odd, so the K\"ahler structure is not generally complex symplectic.  However, it is generally know that the top stratum of $\X_r(G)$ is foliated by complex symplectic manifolds \cite{La2}.  This begs the question, since we just have seen it to be real symplectic globally (on the top stratum): how does the foliated complex symplectic structure relate to the global real symplectic structure?  We leave this interesting question for future work.

\section{Topological Manifolds $\X_r(\SUm{n})$}
\subsection{The Trivial Case}
For $G=\SLm{1}$ both $G$ and $K$ are single points and the corresponding character spaces are single points as well for all values of $r$.

\subsection{Rank 1 Case for all $n$}

\begin{thm}
$\SUm{n}/ \SUm{n}$ is homeomorphic to a closed real ball of real dimension $n-1$.  In particular, they are manifolds with boundary.
\end{thm}
\begin{proof}
Let $\mathfrak{k}=\mathfrak{su}(n)$ be the Lie algebra of $K=\mathrm{SU}(n)$ and $$\mathfrak{h}=\{(\lambda_1,...,\lambda_n)\in \mathbb{R}^n\ | \ \lambda_1+\cdots +\lambda_n=0\}$$ be its Cartan subalgebra.  Then let
$$\mathfrak{h}_+=\{(\lambda_1,...,\lambda_n)\ | \ \lambda_i\geq \lambda_{i+1},\ 1\leq i\leq n-1\}$$ be a closed positive Weyl chamber, and then let $$\mathfrak{a}=\{\lambda \in \mathfrak{h}_+\ |\  \lambda_1-\lambda_n\leq 1 \}$$ be the fundamental alcove.  For any $k\in K$ it eigenvalues may be written $(e^{2\pi i\lambda_1(k)},...,e^{2\pi i\lambda_n(k)})$ where $\lambda(k)=(\lambda_1(k),...,\lambda_n(k))\in \mathfrak{a}$.  The map $k\mapsto \lambda(k)$ induces a homeomorphism $\mathfrak{a}\cong K/K$.  The alcove $\mathfrak{a}$ is topologically a $n-1$ ball with boundary.
\end{proof}

We note that the above argument comes from \cite{AW}.

For example, $\SUm{2}/\SUm{2}\cong [-2,2]$ and $\SUm{3}/\SUm{3}$ is homeomorphic to a disc (see below).

\begin{figure}[ht!]
\begin{center}
\includegraphics[scale=0.5]{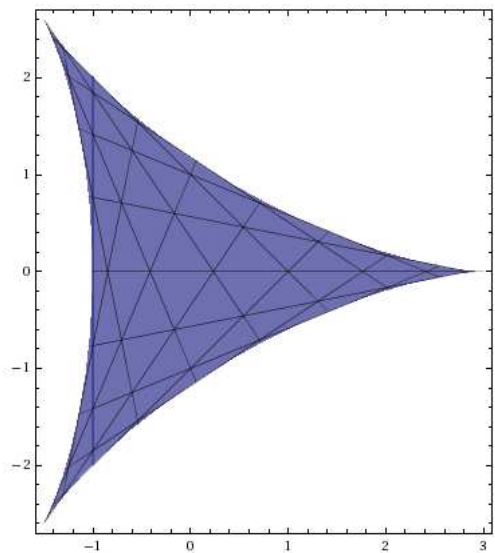}
\caption{$\SUm{3}/\SUm{3}$}\label{su3rank1}
\end{center}
\end{figure}

It is worth describing these cases more explicitly.  The alcove for the case $\SUm{2}$ is determined by $\lambda_1+\lambda_2=0$ and $0\leq \lambda_1-\lambda_2 \leq 1$.  The trace function is real in this case and is equal to $\cos(2\pi\lambda_1)+\cos(2\pi\lambda_2)$ since every matrix in $\SUm{2}$ is conjugate to one in the form 
$$\mathrm{diag}( e^{2 \pi i \lambda_1}, e^{2 \pi i \lambda_1})=\left( \begin{array}{cc} e^{2 \pi i \lambda_1} & 0 \\ 0 & e^{2 \pi i \lambda_2} \end{array} \right),$$
and thus the trace gives a bijection between the alcove and the interval $[-2,2]$.

For $\SUm{3}$ the alcove is determined by $\lambda_1+\lambda_2+\lambda_3=0$ and $\lambda_1\geq \lambda_2\geq \lambda_3 \geq \lambda_1-1$.  In this case each $\SUm{3}$ matrix $\xb$ has its trace given by $e^{2\pi i \lambda_1}+e^{2\pi i \lambda_2}+e^{2\pi i\lambda_3}$.  The coefficients of the characteristic polynomial (which generate the complex invariant ring) are $\tr{\xb}$ and $\tr{\xb^{-1}}=\tr{\overline{\xb}}$, so the real and imaginary parts of the trace function generate the real invariant ring.  Let $p_1=\cos(2 \pi \lambda_1)+\cos(2 \pi \lambda_2)+\cos(2 \pi \lambda_3)$ and $p_2=\sin(2 \pi \lambda_1)+\sin(2 \pi \lambda_2)+\sin(2 \pi\lambda_3)$.  Then $(p_1,p_2)$ gives the isomorphism between the alcove and the region in Figure \ref{su3rank1}.

The region in Figure \ref{su3rank1} can be equally described by \begin{equation}\label{discriminant} 
|\tau|^4-8\Re(\tau^3)+18|\tau|^2-27\leq 0\end{equation} where $\tau=\tr{X}$ (see \cite{G10} page 205).  
This follows since this polynomial is exactly $-1$ times the discriminant of the characteristic polynomial,
that is, the polynomial whose zeros occur exactly when there is a repeated eigenvalue.  
Thus, in Figure \ref{su3rank1}, the boundary corresponds exactly to the conjugation classes
having a representative with a repeated eigenvalue and the interior is where there are three distinct eigenvalues. 
The three corners correspond to the three cubic roots of unity which form the center of $\SUm{3}$, 
by scalar multiples of the identity matrix.

\begin{rem}
On page 168 of \cite{DK}, it is shown that for any simply connected compact Lie group $K$ the Weyl alcove is homeomorphic to $K/K$ where $K$ acts on itself by conjugation $($adjoint action$)$.  Moreover, in this generality any alcove is still homeomorphic to a closed ball.
\end{rem}

\subsection{$\SUm2$ rank 2}

In this section we show $\SUm2^{\times2}/\SUm2$ is a 3-ball and in
the next section we show $\SUm2^{\times3}/\SUm{2}$ is a 6-sphere.
We note that both results appear in \cite{BC}, but our proofs differ.

Let $K=\SUm{2}$. A general element of $K$ will be written as

\[
\left(\begin{array}{cc}
\alpha & \beta\\
-\bar{\beta} & \bar{\alpha}\end{array}\right),\quad\quad\alpha,\beta\in\mathbb{C}\mbox{ such that }|\alpha|^{2}+|\beta|^{2}=1.\]
 We will denote $\alpha=a+ib$, $\beta=c+id$ and we will also view
elements of $K$ as unit quaternions by writing as well\[
g=\alpha+\beta j=a+bi+cj+dk,\quad\quad a^{2}+b^{2}+c^{2}+d^{2}=1,\]
where $\{1,i,j,k\}$ is the usual basis of quaternions over $\mathbb{R}$,
satisfying $k=ij$. Note that $\alpha j=j\bar{\alpha}$ for a complex
$\alpha\in\mathbb{C}$. If $g=a+bi+cj+dk\in K$, we will also use
the notation $a=\Re g$ and $(b,c,d)=\Im g\in\mathbb{R}^{3}$, called
the real and imaginary parts of $g$, respectively. In this setting,
the inverse of $g$ is also the conjugate quaternion \[
g^{-1}=\bar{g}=a-bi-cj-dk=\bar{\alpha}-\beta j,\]
 since elements of $K$ satisfy\[
g\bar{g}=\bar{g}g=a^{2}+b^{2}+c^{2}+d^{2}=1.\]
 We start with the following lemma.

\begin{lem}
\label{pro:K2}There is a one-to-one correspondence between:

\textup{(i)} pairs $(X_{1},X_{2})$ of $\SUm{2}$ matrices such that
$a_{i}=\Re(X_{i})$, $i=1,2$ and $a_{3}=\Re(X_{1}^{-1}X_{2})$, up
to simultaneous conjugation;

\textup{(ii)} triples $(a_{1},a_{2},a_{3})\in[-1,1]^{3}$ such that\begin{equation}
1-a_{1}^{2}-a_{2}^{2}-a_{3}^{2}+2a_{1}a_{2}a_{3}\in[0,1].\label{eq:sigma}\end{equation}

\end{lem}
\begin{proof}
Let $X_{1},X_{2}\in\SUm{2}$ and $X_{3}=X_{1}^{-1}X_{2}$ with $a_{j}=\Re X_{j}$,
$j=1,2,3$. Then, the known formula of Fricke \[
\tr{A^{-1}BAB^{-1}}=\tr{A}^{2}+\tr{B}^{2}+\tr{A^{-1}B}^{2}-\tr{A}\tr{B}\tr{A^{-1}B}-2\]
 (see \cite{G9}, for instance) translates into \[
\Re(X_{1}X_{2}X_{1}^{-1}X_{2}^{-1})=2(a_{1}^{2}+a_{2}^{2}+a_{3}^{2})-4a_{1}a_{2}a_{3}-1\in[-1,1],\]
 and so we obtain the condition (\ref{eq:sigma}). Conversely, let
$(a_{1},a_{2},a_{3})\in[-1,1]^{3}$ satisfy $1-a_{1}^{2}-a_{2}^{2}-a_{3}^{2}+2a_{1}a_{2}a_{3}\in[0,1]$
and consider two $\SUm{2}$ matrices of the form: \begin{equation}
X_{1}=\alpha_{1}=a_{1}+b_{1}i,\quad X_{2}=\alpha_{2}+c_{2}j=a_{2}+b_{2}i+c_{2}j.\label{eq:ansatz}\end{equation}
Let $X_{3}=X_{1}^{-1}X_{2}=\bar{\alpha_{1}}(\alpha_{2}+c_{2}j)=\bar{\alpha_{1}}\alpha_{2}+\alpha_{1}c_{2}j$
and so \[
\Re(X_{3})=\Re(\bar{\alpha_{1}}\alpha_{2})=a_{1}a_{2}+b_{1}b_{2}.\]
Thus, given $a_{1},a_{2}$ and $a_{3}$ in $[-1,1]$, we need to find
real numbers $b_{1},b_{2}$ and $c_{2}$ (also in $[-1,1]$) such
that:\begin{eqnarray}
a_{1}^{2}+b_{1}^{2} & = & 1\nonumber \\
a_{2}^{2}+b_{2}^{2}+c_{2}^{2} & = & 1\label{eq:eqs2}\\
a_{1}a_{2}+b_{1}b_{2} & = & a_{3}.\nonumber \end{eqnarray}
We can solve recursively, obtaining $b_{1}=\sqrt{1-a_{1}^{2}}$, and
let us assume for now that $b_{1}>0$. Then \[
b_{2}=\frac{a_{3}-a_{1}a_{2}}{b_{1}}=\frac{a_{3}-a_{1}a_{2}}{\sqrt{1-a_{1}^{2}}}\]
 and \[
c_{2}=\sqrt{1-a_{2}^{2}-b_{2}^{2}}=\frac{\sqrt{1-a_{1}^{2}-a_{2}^{2}-a_{3}^{2}+2a_{1}a_{2}a_{3}}}{\sqrt{1-a_{1}^{2}}}\]
 is a solution with $b_{1},b_{2},c_{2}$ real because of condition
(\ref{eq:sigma}). In the case $b_{1}=0$ we just take $c_{2}=0$
and $b_{2}=\sqrt{1-a_{2}^{2}}>0$. To complete the proof, we just
need to show that any two pairs $(X_{1},X_{2})$ of the form (\ref{eq:ansatz})
verifying the system (\ref{eq:eqs2}) are conjugate. This is a consequence
of the fact that the solution provided is unique up to the choices
of signs in the square roots. It is then a simple computation to show
that different choices give conjugate pairs.
\end{proof}

\begin{rem}
\label{rem:JW}Let $\theta_{1},\theta_{2}$ and $\theta_{3}$ be three
real numbers in $[0,1]$, related to the variables $a_{i}$ by $a_{i}=\cos(\pi\theta_{i})$,
$i=1,...,3$. Then, as shown in \cite{JW} both conditions of the
lemma are equivalent to \begin{eqnarray}
\theta_{i}+\theta_{j}-\theta_{k} & \geq & 0,\qquad\textrm{for all}\ i,j,k\in\{1,2,3\}\nonumber \\
\theta_{1}+\theta_{2}+\theta_{3} & \leq & 2.\label{eq:ineq}\end{eqnarray}
Note that these inequalities define a 3-dimensional compact tetrahedron.
\end{rem}

\begin{thm}
Let $K=\SUm{2}$. The topological space $K^{\times2}/K$ is homeomorphic to a 3 dimensional closed ball. 
\end{thm}
\begin{proof}
First, observe that the map $\phi:[0,1]^{3}\to[-1,1]^{3}$ sending
the coordinates $\theta_{i}$ to the coordinates $a_{i}$ as above
is a homeomorphism. Since the domain of the inequalities (\ref{eq:ineq})
is a tetrahedron, hence homeomorphic to a 3-ball, the condition $1-a_{1}^{2}-a_{2}^{2}-a_{3}^{2}-2a_{1}a_{2}a_{3}\in[0,1]$
also defines a 3-ball by the lemma. The result then follows from the
fact that any 3 numbers $(a_{1},a_{2},a_{3})\in[-1,1]^{3}$ satisfying
$1-a_{1}^{2}-a_{2}^{2}-a_{3}^{2}-2a_{1}a_{2}a_{3}\in[0,1]$ correspond
in a one-to-one fashion to a pair of elements $(g_{1},g_{2})\in K^{2}$
up to simultaneous conjugation, as shown in the lemma. 
\end{proof}
Note that the homeomorphism $\phi$ sends the boundary of tetrahedron
to a space which is a 2 sphere smooth except at the 4 points $(1,1,1),(1,-1-1),(-1,1,-1)$
and $(-1,-1,1)$.

\begin{figure}[ht!]

\begin{centering}
\includegraphics[scale=0.5]{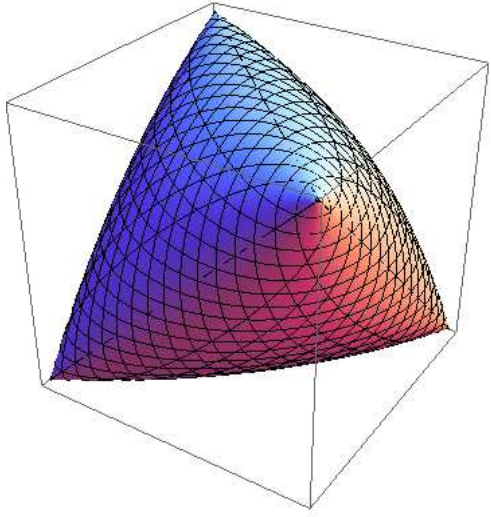} 
\par\end{centering}

\caption{$\SUm{2}^{\times2}/\SUm{2}$}

\end{figure}

\subsection{$\SUm{2}$ rank 3}

To describe $K^{\times3}/K$ for $K=\SUm{2}$ let us consider the
map:\begin{eqnarray*}
\mathcal{T}:K^{\times3}/K & \to & [-1,1]^{\times6}\\
\left[(X_{1},X_{2},X_{3})\right] & \mapsto & (\Re A_{1},...,\Re A_{3},\Re(A_{1}^{-1}A_{2}),...,\Re(A_{2}^{-1}A_{3})).\end{eqnarray*}
 Given a pair of $\SUm{2}$ matrices $(X_{1},X_{2})$ let us define
the following expression, which represents the cosine of the angle
between the vectors corresponding to $X_{i}$ and is clearly invariant
under simultaneous conjugation: \[
l_{12}=\frac{\Re(X_{1}^{-1}X_{2})-\Re(X_{1})\Re(X_{2})}{||\Im X_{1}||\ |||\Im X_{2}||}.\]

\begin{prop}
\label{K3}The image of $\mathcal{T}$ is defined by the following
4 inequalities:

\textup{(i)} For every two distinct indices $i,j\in\{1,2,3\}$\[
1-a_{i}^{2}-a_{j}^{2}-a_{ij}^{2}+2a_{i}a_{j}a_{ij}\in[0,1],\quad\quad\mbox{and }\]

\textup{(ii)} \[
1-a_{12}^{2}-a_{13}^{2}-a_{23}^{2}+2a_{12}a_{13}a_{23}\in[0,1].\]
 Conversely, if $a_{1},a_{2},a_{3},a_{12},a_{13},a_{23}$ are six
real numbers in $[-1,1]$ satisfying those 4 conditions, there are
three $\SUm{2}$ matrices $X_{1},X_{2}$ and $X_{3}$ such that $\mathcal{T}(X_{1},X_{2},X_{3})=(a_{1},...,a_{23})$.

Moreover, such triples $(X_{1},X_{2},X_{3})$ are unique up to conjugation
if and only if\begin{equation}
1-l_{12}^{2}-l_{13}^{2}-l_{23}^{2}+2l_{12}l_{13}l_{23}=0;\label{eq:delta}\end{equation}
 if this equation does not hold there are exactly 2 such triples. 
\end{prop}
\begin{proof}
From Proposition \ref{pro:K2} we know that the image of $\mathcal{T}$
satisfies the inequalities in (i). The condition (ii) is not difficult
to show. Conversely, to find a triple $(X_{1},X_{2},X_{3})$ in the
image, suppose first that the triple is irreducible, and one of the
matrices is not $\pm\id$. So, after reordering the matrices, we may
assume $\det([X_{1},X_{2}])\neq0$, and $X_{1}\neq\pm\id$. Let us
consider the following invariants \begin{eqnarray*}
r_{jk} & = & \Re(X_{j}X_{k})-\Re X_{j}\Re X_{k}=a_{jk}-a_{j}a_{k}\\
s_{jk} & = & \frac{1-\Re(X_{j}X_{k}X_{j}^{-1}X_{k}^{-1})}{2}=1-a_{i}^{2}-a_{j}^{2}-a_{ij}^{2}+2a_{i}a_{j}a_{ij}\\
t_{jkl} & = & \frac{1}{r_{11}r_{22}r_{33}}\left|\begin{array}{ccc}
r_{11} & r_{12} & r_{13}\\
r_{21} & r_{22} & r_{23}\\
r_{31} & r_{32} & r_{33}\end{array}\right|=1-l_{12}^{2}-l_{13}^{2}-l_{23}^{2}+2l_{12}l_{13}l_{23}.\end{eqnarray*}
 Then, a direct computation shows that the following matrices form
a solution to $\mathcal{T}(X_{1},X_{2},X_{3})=(a_{1},...,a_{23})$.\[
X_{1}=\left(\begin{array}{cc}
a_{1}+ib_{1} & 0\\
0 & a_{1}-ib_{1}\end{array}\right),\qquad X_{2}=\left(\begin{array}{cc}
a_{2}+ib_{2} & id_{2}\\
id_{2} & a_{2}-ib_{2}\end{array}\right),\]
 \[
X_{3}=\left(\begin{array}{cc}
a_{3}+ib_{3} & c_{3}+id_{3}\\
-c_{3}+id_{3} & a_{3}-ib_{3}\end{array}\right),\]
 with\[
b_{1}=\sqrt{1-a_{1}^{2}}>0,\qquad b_{k}=-\frac{r_{1k}}{b_{1}},\qquad k=2,3,\]
 \[
d_{2}=\frac{\sqrt{s_{12}}}{b_{1}},\qquad d_{3}=\frac{r_{23}b_{1}^{2}+r_{12}r_{13}}{-d_{2}b_{1}^{2}},\]
 \[
c_{3}=\pm\sqrt{\frac{t_{123}}{s_{12}}}.\]
 Note that $c_{3}=0$ if and only if the two triples are the same,
and this is verified precisely when $t_{123}=0$. The cases when the
above solution is not defined are when $b_{1}=0$, $d_{2}=0$ or $s_{12}=0$.
In this case, the pair $(X_{1},X_{2})$ is not irreducible and corresponds
to $s_{12}=0$. If one of the other pairs are irreducible, we can
just relabel the matrices. If all three pairs are reducible, then
the triple is simultaneously diagonalizable, and it is easy to find
a solution.

To finish the proof, we just need to observe that the solution is
not unique, but up to conjugation there is at most two possible solutions,
each one corresponding to a choice of sign on $c_{3}$. So there is
exactly one solution, up to conjugation, if $t_{123}=0$ (which corresponds
to the Equation (\ref{eq:delta})) and exactly two, when $t_{123}\neq0$. 
\end{proof}

\begin{prop}
\label{pro:star}The image of $\mathcal{T}$ is a compact convex set
with non-empty interior. Thus it is homeomorphic to a $6$-dimensional
ball. 
\end{prop}
\begin{proof}
Consider the homeomorphism \begin{eqnarray*}
\phi:[0,1]^{6} & \to & [-1,1]^{6}\\
(\theta_{1},...,\theta_{6}) & \to & (a_{1},...,a_{6})\end{eqnarray*}
given by $a_{j}=cos(\pi\theta_{j})$, $j=1,...,6$. Let us make the
identifications $a_{4}=a_{12}$, $a_{5}=a_{13}$ and $a_{6}=a_{23}$,
and let $P=\phi^{-1}\left(\mathcal{T}(\X_3(K))\right)\subset[0,1]^{6}$. Then $\phi$ restricts to a homeomorphism
between $P$ and the image of $\mathcal{T}$. Because of Remark \ref{rem:JW},
the set $P$ is defined by 16 inequalities of the form of Equations
\ref{eq:ineq}. Indeed, each condition in Proposition \ref{K3} involves only three variables, corresponding to
four such inequalities defining a tetrahedron in the space defined by those variables. Therefore $P$ is the intersection of 4
compact convex spaces, each being the Cartesian product of a tetrahedron
and the cube $[0,1]^{3}$. So, $P$ is itself a compact convex set
with non-empty interior.  Thus, $P$ is a closed 6-ball and the same holds for the image of $\mathcal{T}$. 
\end{proof}

\begin{thm}
The space $\X_3(\SUm{2})$ is homeomorphic to $S^{6}$. 
\end{thm}
\begin{proof}
Let $B$ denote the image of $\mathcal{T}$, and $B^{o}$ its interior. According
to Proposition \ref{K3} we can write $K^{\times3}/K$ as $B_{+}\sqcup B_0\sqcup B_{-}$
where $B_{\pm}$ correspond to the $\pm$ sign of $c_{3}$ and $V$
correspond to $c_{3}=t_{123}=0$. Then, because of the form of the
solutions, $\mathcal{T}|_{B_{\pm}}\to B^{o}$ are homeomorphisms onto an open
6-ball. Moreover, since the solutions extend to the boundary $\partial B$,
$\mathcal{T}|_{B_0}\to\partial B$ is also a homeomorphism onto a 5-sphere. So
$K^{\times3}/K$ is obtained from the gluing of 2 6-balls along a
5-sphere, and is therefore homeomorphic to a 6-sphere.
\end{proof}

\subsection{$\SUm3$ rank 2}
Our goal is now to establish
\begin{thm}\label{s8thm}
$\X_2(\SUm3)$ is homeomorphic to an 8-sphere.
\end{thm}

The proof of this theorem is surprisingly long, since unlike the case of $\SUm{2}^{\times 3}/\SUm{2}$, we do not have an explicit slice.  We break the discussion into a series of subsections.  Throughout $K=\SUm{3}$ and $G=\SLm{3}$.

\subsubsection{Obtaining real coordinates from complex coordinates.}

In \cite{La1} the following theorem is established.

\begin{thm} 
\begin{enumerate}
\item[]
\item[$(i)$] $\C[\X_2(\SLm{3})]$ is minimally generated by the nine affine coordinate functions
\begin{align*}
\mathcal{G}=&\{\tr{\xb_1},\tr{\xb_2},\tr{\xb_1\xb_2},\tr{\xb_1^{-1}},\tr{\xb_2^{-1}},\tr{\xb_1\xb_2^{-1}},\\ &\tr{\xb_2\xb_1^{-1}},\tr{\xb_1^{-1}\xb_2^{-1}},\tr{\xb_1\xb_2\xb_1^{-1}\xb_2^{-1}}\}.
\end{align*}
\item[$(ii)$] The eight elements in $\mathcal{G}\backslash \{\tr{\xb_1\xb_2\xb_1^{-1}\xb_2^{-1}}\}$ are a maximal algebraically independent subset.  Therefore, they are local parameters, since the Krull dimension is $8$.
\item[$(iii)$] $\tr{\xb_1\xb_2\xb_1^{-1}\xb_2^{-1}}$ satisfies a monic (degree 2) relation over the algebraically independent generators.  It generates the ideal.
\item[$(iv)$] This relation is of the form $t^2-Pt+Q$ where $$P=\tr{\xb_1\xb_2\xb_1^{-1}\xb_2^{-1}}+\tr{\xb_2\xb_1\xb_2^{-1}\xb_1^{-1}}$$ and $$Q=\tr{\xb_1\xb_2\xb_1^{-1}\xb_2^{-1}}\tr{\xb_2\xb_1\xb_2^{-1}\xb_1^{-1}}$$
\end{enumerate}

\end{thm}

We note that \cite{La1} provides explicit formulas for $P$ and $Q$ in terms of the first eight elements of $\mathcal{G}$.

Said differently, $\X_2(\SLm{3})\to \C^8$ is a branched double cover, submersive off the branching locus.  What is not discussed in \cite{La1} is how the two preimages of any given image point are related.  We now address this as it becomes relevant.

To simplify matters, make the following notational changes:

\begin{center}
\begin{tabular}{ll}
$t_{(1)}=\tr{\xb_1}$ & $t_{(-1)}=\tr{\xb_1^{-1}}$\\
$t_{(2)}=\tr{\xb_2}$& $t_{(-2)}=\tr{\xb_2^{-1}}$\\
$t_{(3)}=\tr{\xb_1\xb_2}$& $t_{(-3)}=\tr{\xb_1^{-1}\xb_2^{-1}}$\\
$t_{(4)}=\tr{\xb_1\xb_2^{-1}}$& $t_{(-4)}=\tr{\xb_1^{-1}\xb_2}$\\
$t_{(5)}=\tr{\xb_1\xb_2\xb_1^{-1}\xb_2^{-1}}$& $t_{(-5)}=\tr{\xb_2\xb_1\xb_2^{-1}\xb_1^{-1}}$.
\end{tabular}
\end{center}

We will use this notation throughout this section.

In these terms, the branched double covering map is $$\mathcal{T}=(\ti{1},\ti{-1},...,\ti{4},\ti{-4}):\X_2(\SLm{3})\to \C^8.$$ 

\begin{prop}
The involution of $\R_2(\SLm{3})$ given by $(\xb_1,\xb_2)\mapsto(\xb_1^\mathsf{t},\xb_2^\mathsf{t})$ descends to an involution of $\X_2(\SLm{3})$ and the branching locus of $\mathcal{T}$ is exactly the fixed point set of this involution.
\end{prop}
\begin{proof}
Clearly, since all the generators $\ti{\pm k}$ for $1\leq k\leq 4$ are in terms of no more than two generic matrices, the transpose mapping fixes the traces.  For instance, $$\tr{\xb_1^\mathsf{t}(\xb_2^\mathsf{t})^{-1}}=\tr{\xb_1^\mathsf{t}(\xb_2^{-1})^{\mathsf{t}}}=\tr{(\xb_2^{-1}\xb_2)^\mathsf{t}}\tr{\xb_2^{-1}\xb_1}=\tr{\xb_1\xb_2^{-1}}.$$  Thus the first 8 coordinates of any preimage of the projection mapping are fixed by the transpose involution.

However, the values of $\ti{5}$ and $\ti{-5}$ are switched.  This follows from the following calculation: 
\begin{align*}
\tr{\xb_1^\mathsf{t}\xb_2^\mathsf{t}(\xb_1^\mathsf{t})^{-1}(\xb_2^\mathsf{t})^{-1}}&=\tr{(\xb_2^{-1}\xb_1^{-1}\xb_2\xb_1)^\mathsf{t}}\\ &=\tr{\xb_2^{-1}\xb_1^{-1}\xb_2\xb_1}=\tr{\xb_2\xb_1\xb_2^{-1}\xb_1^{-1}}.
\end{align*}

Since $\ti{5}$ and $\ti{-5}$ are the two possible values of the ninth coordinate of any preimage of the mapping $\mathcal{T}$, the preimages are permuted by the transpose.  Since the branching locus is defined to be the points where $\ti{5}=\ti{-5}$, these coincide with the representations whose semi-simplification is conjugate to a transpose invariant representation.
\end{proof}

In general, $$\tr{\wb^{-1}}=\tr{\overline{\wb}^\mathsf{t}}=\tr{\overline{\wb}}=\overline{\tr{\wb}}$$ whenever $\wb$ is unitary.  Therefore, for any unitary matrix $\wb$, $\tr{\wb+\wb^{-1}}=\tr{\wb}+\overline{\tr{\wb}}=2\Re(\tr{\wb})$.  Thus restricting $P$ and $Q$ to $\SUm{3}^{\times 2}$ we have $P=2\Re(\tr{\xb_1\xb_2\xb_1^{-1}\xb_2^{-1}})$ and $Q=|\tr{\xb_1\xb_2\xb_1^{-1}\xb_2^{-1}}|^2$.  In particular they are real, and thus $$\tr{\xb_1\xb_2\xb_1^{-1}\xb_2^{-1}}=\frac{P\pm\sqrt{P^2-4Q}}{2} \in \mathbb{R}\sqcup i\mathbb{R}$$ is well defined for all unitary representations (we do not need to take a branch cut!).

Consider the change of variables on the algebraically independent parameters: $$\left(\tr{\wb},\tr{\wb^{-1}}\right)\mapsto \left(\frac{\tr{\wb}+\tr{\wb^{-1}}}{2},\frac{\tr{\wb}-\tr{\wb^{-1}}}{2i}\right).$$

When $\wb$ is unitary, these coordinates are real.

We change the dependent coordinate $\tr{\xb_1\xb_2\xb_1^{-1}\xb_2^{-1}}$ to $$u_{(5)}=\frac{\tr{\xb_1\xb_2\xb_1^{-1}\xb_2^{-1}}-\tr{\xb_2\xb_1\xb_2^{-1}\xb_1^{-1}}}{2i}=\frac{2\tr{\xb_1\xb_2\xb_1^{-1}\xb_2^{-1}}-P}{2i}.$$

Then let $u_{(k)}=\frac{\ti{k}+\ti{-k}}{2}$ and $u_{(-k)}=\frac{\ti{k}-\ti{-k}}{2i}$.  
It is clear that $\ti{k}=u_{(k)}+iu_{(-k)}$ and $u_{(k)}|_{K^{\times 2}}=\Re(\ti{k})$ and $u_{(-k)}|_{K^{\times 2}}=\Im(\ti{k}).$

Since $\ti{5}=iu_{(5)}+P/2$ we conclude that 
\begin{align*}
\ti{5}^2-P\ti{5}+Q&=(iu_{(5)}+P/2)^2-P(iu_{(5)}+P/2)+Q\\
&=-u_{(5)}^2+P^2/4+iPu_{(5)}-P^2/2-iPu_{(5)}+Q\\
&=-u_{(5)}^2+(Q-P^2/4),
\end{align*}
and thus $u_{(5)}=\pm\sqrt{Q-P^2/4}$.  In particular, the discriminant satisfies $\Re(P^2-4Q)\leq 0$ and $\Im(P^2-4Q)=0$ when restricted to $K^{\times 2}$ since $u_{(5)}$ is real on $K^{\times2}$.  Moreover, $P^2-4Q=0$ if and only if $u_{(5)}=0$ if and only if $t_{(5)}=P/2$.  Since $u_{(5)}=\Im{\ti{5}}$ for a unitary representation, any such representation is in the branch locus if and only if $\tr{\xb_1\xb_2\xb_1^{-1}\xb_2^{-1}}$ is real.

\subsubsection{Some topological consequences of real coordinates}
Making the change of variables from the $\ti{k}$'s to the $u_{(k)}$'s does not change $\C[\X_2(G)]$ since we are working over $\C$.  Consequently, since the trace mapping $\mathcal{T}\times \ti{5}:\X_2(G)\hookrightarrow \C^9$ is injective, the mapping $$\mathcal{U}\times u_{(5)}=(u_{(1)},u_{(-1)},...,u_{(4)},u_{(-4)},u_{(5)})$$ defines an injection $\X_2(G)\hookrightarrow \C^9$ equally well.  By Lemma \ref{unitarylemma}, $\X_2(K)\hookrightarrow \X_2(G)$ is also injective which implies $\X_2(K)\hookrightarrow\mathbb{R}^9$ is injective because the $u_{(i)}$'s are real valued upon restriction to $\X_2(K)$.  Clearly $\mathcal{U}\times u_{(5)}$ is continuous since it is polynomial.  In fact it is a closed mapping since $K^{\times 2}/K$ is compact and $\mathbb{R}^9$ is Hausdorff.  Therefore, it defines a homeomorphism onto its image; of real dimension 8 since the first eight generators are algebraically independent.

Moreover, since it is the continuous image of a compact path-connected space (compact path-connected quotients of compact path-connected spaces are themselves compact and path-connected), its image is compact and path-connected as well.  The discriminant locus $P^2-4Q=0$ separates the image into three disjoint pieces since the plane $u_{(5)}=0$ separates $\mathbb{R}^9$.  The involution which sends a unitary representation to its transpose defines a homeomorphism between the two open subspaces $B_+=K^{\times 2}/K \cap \{u_{(5)}>0\}$ and $B_-=K^{\times 2}/K \cap \{u_{(5)}<0\}$.  Since there are abelian unitary representations and they have trivial commutator these representations satisfy $u_{(5)}=0$ and thus $B_0=K^{\times 2}/K \cap \{u_{(5)}=0\}$ is non-empty.

It is shown in \cite{La1} that the reducible characters (points in $\X_2(G)$ that correspond to extended equivalence classes of reducible representations) are contained in the branching locus $P^2-4Q=0$.  Thus the reducible unitary characters are contained in $B_0$ and thus are fixed by the transpose involution.  In fact we can describe $B_0$ as the union of the set of reducible characters with the set of irreducible characters that are transpose fixed.  Thus, $B_-$ and $B_+$ contain only irreducible characters that are not transpose fixed.  They are then smooth (in particular topological manifolds), and the algebraically independent coordinates $u_{(1)},...,u_{(-4)}$ define global parallelizable parameters.

\subsubsection{A partial description of the semi-algebraic structure.}\label{semialgebraic}
Let $\mathcal{U}=(u_{(1)},...,u_{(-4)})$.  Then $\mathcal{U}(B_+)$ is a bounded open smooth path-connected subspace of a hyper-cube $[a_1,b_1]\times \cdots \times [a_8,b_8]\subset \mathbb{R}^8$.  From our description of the alcove for $\SUm{3}/\SUm{3}$, we can see that $(u_{(k)},u_{(-k)})\in [\frac{-3}{2},3]\times [\frac{-3\sqrt{3}}{2},\frac{3\sqrt{3}}{2}]$ for $1\leq k\leq 4$.  Moreover, $u_{(5)}\in [\frac{-3\sqrt{3}}{2},\frac{3\sqrt{3}}{2}]$.  It is not hard to find representations that realize the extreme values for each of these 9 coordinates.  Thus this is the smallest hyper-cube possible.  Moreover, a further necessary condition is that $-27\leq P^2-4Q\leq 0$.  This follows from that fact that $u_{(5)}=\pm\sqrt{Q-P^2/4}.$  
Any abelian unitary representation will make $P^2-4Q=0$, but letting
$$\xb_1=\left(
\begin{array}{lll}
 0 & 1 & 0 \\
 0 & 0 & 1 \\
 1 & 0 & 0
\end{array}
\right), \quad
\xb_2=\left(
\begin{array}{lll}
 e^{-\frac{2 i \pi }{3}} & 0 & 0 \\
 0 & e^{\frac{2 i \pi }{3}} & 0 \\
 0 & 0 & 1
\end{array}
\right),$$
we get a representation that has the first 8 coordinates 0.  However, its commutator is
$\xb_1\xb_2\xb_1^{-1}\xb_2^{-1}=\mathrm{diag}(e^{\frac{2 i \pi }{3}}, e^{\frac{2 i \pi }{3}} , e^{\frac{2 i \pi }{3}})$ 
and the imaginary part of the trace of this matrix is $\frac{3\sqrt{3}}{2}$ making $P^2-4Q$ equal to $-27$. 
This representation is not transpose invariant, and so is irreducible and moreover,
corresponds to the ``center'' of $B_+$.

Let $\X_1\subset[\frac{-3}{2},3]\times [\frac{-3\sqrt{3}}{2},\frac{3\sqrt{3}}{2}]$ be the homeomorphic image given by the trace map of the Weyl alcove of $\mathfrak{su}(3)$, as described earlier (see Figure 1).

Then for $\xb_{1}$ as above and letting
\[
\xb_{2}=\left(
\begin{array}{lll}
e^{\frac{i(2\alpha+\beta)}{3}} & 0 & 0\\
0 & e^{\frac{i(\beta-\alpha)}{3}} & 0\\
0 & 0 & e^{\frac{-i(\alpha+2\beta)}{3}}\end{array}
\right)
\]
with $\alpha,\beta\in\mathbb{R}$ we have a representation satisfying
$\tr{X_{1}X_{2}X_{1}^{-1}X_{2}^{-1}}=e^{i\alpha}+e^{i\beta}+e^{-i(\alpha+\beta)},$
which is the general form of an element of $\X_{1}$. This shows that
the trace of the commutator defines a surjection $\kappa:\X_{2}(K)\to\X_{1}$.
In fact, for $X_1,X_2\in G=\SLm{3}$, 
we have an extension of $\kappa$, denoted by $\tilde{\kappa}:\X_{2}(G)\to\mathbb{C}$,
defined in the same fashion by $\tilde{\kappa}((X_{1},X_{2}))=\tr{X_{1}X_{2}X_{1}^{-1}X_{2}^{-1}}.$
It is again surjective by an analogous argument (using now $\alpha,\beta\in\mathbb{C}$)
and forms a commutative diagram
\begin{eqnarray*}
\X_{2}(K) & \hookrightarrow & \X_{2}(G)\\
\kappa\downarrow &  & \downarrow\tilde{\kappa}\\
\X_{1} & \hookrightarrow & \mathbb{C}.
\end{eqnarray*}
For proof of a more general result see Proposition 6 in \cite{La2}.

Using Inequality \eqref{discriminant} and letting $\tau=t_{(5)}$,
we have $P=2\Re(\tau)$ and $Q=|\tau|^{2}$, and $B_{+}$ is characterized
by $\Im(\tau)>0$ in $\X_{2}(K)$ which, as above, implies that $P^{2}-4Q<0$.
As suggested by the referee, in view of $2\Re(\tau^{3})=P^{3}-3PQ$,
Inequality \eqref{discriminant} becomes \[
\Delta:=Q^{2}+12PQ+18Q-4P^{3}-27\leq0.\]

We can thus say that $B_{+}$ is a semi-algebraic subset of the semi-algebraic
set \begin{equation}
\mathcal{S}_{+}:=\{(x_{1},...,x_{8})\in\X_{1}^{\times4}\ |\ \Delta\leq0\text{ and }P^{2}-4Q<0\}.
\label{semialgset}
\end{equation}

%Then for $\xb_1$ as above and letting $$\xb_2=\left(
%\begin{array}{lll}
% e^{\frac{i(2\alpha+\beta)}{3}} & 0 & 0 \\
% 0 & e^{\frac{i(\beta-\alpha)}{3}} & 0 \\
% 0 & 0 & e^{\frac{-i(\alpha +2\beta)}{3}}
%\end{array}
%\right)$$ we have a representation satisfying
%$\tr{X_1X_2X_1^{-1}X_2^{-1}}=e^{i \alpha}+e^{i
%\beta}+e^{-i(\alpha+\beta)},$ which is the general form of an element of $\X_1$.  This shows that the trace of the commutator defines a surjection $\kappa:\X_2(K)\to \X_1$.

%Using Inequality \eqref{discriminant} and letting $\tau=t_{(5)}$, we have $P=2\Re(\tau)$ and $Q=|\tau|^2$, 
%and $B_+$ is characterized by $\Im(\tau)>0$ in $\X_2(K)$ which, as above, implies that $P^2-4Q<0$.  
%As suggested by the referee, in view of $2\Re(\tau^3)=P^3-3PQ$, Inequality \eqref{discriminant} becomes
%$$\Delta:=Q^2+12PQ+18Q-4P^3-27\leq 0.$$

%We can thus say that $B_+$ is a semi-algebraic subset of the semi-algebraic set 
%\begin{equation}\label{semialgset} 
%\mathcal{S}_+:=\{(x_1,...,x_8)\in \X_1^{\times 4}\ |\ \Delta \leq 0 \text{ and } P^2-4Q< 0\}.
%\end{equation}

On the other hand, taking any 8-tuple $(x_{1},...,x_{8})\in\mathcal{S}_{+}$,
using the values of $P$ and $Q$ determined by this tuple, and the
inequality $\Delta\leq0$ we determine a point $\tau_{0}\in\X_{1}\subset\mathbb{C}$,
such that $P=2\Re(\tau_{0})$ and $Q=|\tau_{0}|^{2}$. Then, by the
surjectivity of $\tilde{\kappa}$, $\tilde{\kappa}^{-1}(\tau_{0})$
contains between one and two elements in $\X_{2}(G)$ that correspond to $(x_{1},...,x_{8})$,
that is at most two classes $[(X_{1},X_{2})]$ (but at least one), such that $\mathcal{U}(X_{1},X_{2})=(x_{1},...,x_{8})$.
However, it is not clear that this class contains a representative
$(X_{1},X_{2})\in K^{2}$. If this happens, the condition $P^{2}-4Q<0$
guarantees that it represents a point in $B_{+}$ which would imply
$\mathcal{S}_{+}=B_{+}$. Otherwise, we can only say, as stated before,
that $B_{+}$ is a semi-algebraic subset of $\mathcal{S}_{+}$. This is not trivial to check and will not be needed in the sequel.

%On the other hand, taking any 8-tuple $(x_1,...,x_8)\in \mathcal{S}_+$, we determine a point in $\kappa(\X_2(K))\cong \X_1$ from $\Delta\leq 0$.  Call it $\tau_0$.  Then $\kappa^{-1}(\tau_0)$ contains at most two elements in $\X_2(K)$ that correspond to $(x_1,...,x_8)$.  If we can find $(X_1, X_2)$ so $\mathcal{U}(X_1,X_2)=(x_1,...,x_8)$, that is a local section, then the condition $P^2-4Q<0$ guarantees that it represents a point in $B_+$ which would imply $\mathcal{S}_+=B_+$.  Similarly we could describe $B_-$ and $B_0$.  However, we do not have such a section.  It does seem reasonable that a section exists and thus motivates the conjecture that \eqref{semialgset} is a complete description of $B_+$.

\subsubsection{A partial product decomposition.}
The first matrix in the pair $(\xb_1,\xb_2)$ can always be taken to be diagonal, 
since all unitary matrices are diagonalizable.  Let it be denoted 
$X_1=\mathrm{diag} (\lambda _1 , \lambda _2 ,\lambda _3)$, such that $\lambda _1 \lambda _2 \lambda _3=1$.  
Moreover, if the orbit of $\xb_1$ is on the boundary of the alcove it has a repeated root, which occurs if and only if $$\left(\lambda _1-\lambda _2\right) \left(\lambda _1-\lambda _3\right) \left(\lambda _2-\lambda_3\right)=0.$$ This implies that the $K$-orbit of the pair $(\xb_1,\xb_2)$ cannot be of dimension greater than 7 (1 dimension for $\xb_1$ and no more than 6 for $\xb_2$).  Consequently, such a representation cannot be irreducible since all irreducible representation have 8 dimensional orbits.  Therefore, let $(\X_1)^o$ correspond to the interior of the Weyl alcove. 
Then there remains only a torus action, $T=S^1\times S^1$, on $X_2$ which preserves the form of $X_1$.   
Let $\xb_2=\left( x_{ij} \right)_{i,j=1,2,3}$ so that the torus action on $\SUm{3}$, denoted by 
$\mathrm{diag}( \mu _1,\mu _2 ,\mu _3)$, gives 
$$X_2\mapsto \left(
\begin{array}{lll}
 x_{11} & \frac{\mu _1 x_{12}}{\mu _2} & \frac{\mu _1 x_{13}}{\mu _3} \\
 \frac{\mu _2 x_{21}}{\mu _1} & x_{22} & \frac{\mu _2 x_{23}}{\mu _3} \\
 \frac{\mu _3 x_{31}}{\mu _1} & \frac{\mu _3 x_{32}}{\mu _2} & x_{33}
\end{array}
\right).$$
Define
$$S=\{ [X_2]\in\SUm{3}/T\ |\  \left(x_{12} x_{23} x_{31}-x_{13} x_{21} x_{32}\right)\not=0\}.$$ 
Clearly the polynomial $x_{12} x_{23} x_{31}-x_{13} x_{21} x_{32}$ is invariant and the set $S$ is well defined.  

\begin{prop} \label{productdecomp}
There is a homeomorphism $$B_+\sqcup B_-\cong (\X_1)^o\times\{ [X]\in\SUm{3}/T\ |\  \left(x_{12} x_{23} x_{31}-x_{13} x_{21} x_{32}\right)\not=0\}.$$  
\end{prop}

\begin{proof}
As already noted, $B_+\sqcup B_-$ contains only irreducibles and $(\X_1)^o\times \SUm{3}/T$ contains all irreducibles, so $$B_+\sqcup B_-\hookrightarrow (\X_1)^o\times \SUm{3}/T \hookrightarrow \X_2(K).$$

To exclude the set $B_0\cap\big( (\X_1)^o\times \SUm{3}/T\big)$ from $(\X_1)^o\times \SUm{3}/T$, and by such define the image of $(\X_1)^o\times \SUm{3}/T\hookrightarrow B_+\sqcup B_-$, we need to ensure that there are two roots to the equation $t^2-Pt+Q$.  This is satisfied if $\tr{\xb_1\xb_2\xb_1^{-1}\xb_2^{-1}}-\tr{\xb_2\xb_1\xb_2^{-1}\xb_1^{-1}}\not=0$.  Computing this expression for $\xb_1$ given above, and imposing the condition that $\mathrm{det}(\xb_2)=1$ we conclude
$$-\left(\lambda _1-\lambda _2\right) \left(\lambda _1-\lambda _3\right) \left(\lambda _2-\lambda
3\right) \left(x_{12} x_{23} x_{31}-x_{13} x_{21} x_{32}\right)\not=0.$$
However, for $X_1\in (\X_1)^o$ the expression $$\left(\lambda _1-\lambda _2\right) \left(\lambda _1-\lambda _3\right) \left(\lambda _2-\lambda_3\right)\not=0.$$ Thereofore, the semi-algebraic subset $$S=\{[\xb]\in \SUm{3}/T\ |\ x_{12}x_{23}x_{31}-x_{21}x_{32}x_{13}\not=0\},$$
defines a subset where $(\X_1)^o\times S$ is bijectively in correspondence to $B_+\sqcup B_-$.  This bijection is given by the mapping $\mathcal{U}$, so is continuous and proper.  Thus it is a homeomorphism and $B_+\sqcup B_-$ is topologically a product.  Now, since $B_+$ and $B_-$ are disjoint and homeomorphic and their union is a product, this implies each subset $B_+$ and $B_-$ is a non-trivial topological product as well.
\end{proof}

\begin{rem}
One can also consider the complex torus action $T_{\C}=\C^*\times \C^*$ on $\SLm{3}$ and ask about the quotient $\SLm{3}\aq T_{\C}$.  The collection of minors on $X_2$, 
\begin{align*}
&\{m_1=x_{11}, m_2=x_{22}, m_3=x_{33}, m_{-1}=x_{23} x_{32}-x_{22}x_{33},\\ 
& m_{-2}=x_{13} x_{31}-x_{11} x_{33}, m_{-3}=x_{12} x_{21}-x_{11}x_{22}, m_4=x_{12} x_{23} x_{31}\},
\end{align*} 
are all invariant under the complex torus.  Moreover they satisfy the following relation:
\begin{align*}
0=&-m_2^2 m_3^2 m_1^2+m_{-1} m_2 m_3 m_1^2+m_{-3} m_2 m_3^2 m_1-m_{-2} m_{-1} m_2 m_1+\\
&m_{-2} m_2^2m_3 m_1-m_{-3} m_{-1} m_3 m_1-m_{-1} m_4 m_1+2 m_2 m_3 m_4 m_1-m_4^2+\\
&m_{-3} m_{-2} m_{-1}-m_{-3}m_{-2} m_2 m_3-m_{-2} m_2 m_4-m_{-3} m_3 m_4+m_4.\end{align*}
It can be shown by solving these equations that $$\mathcal{M}=(m_1,m_2,m_3,m_{-1},m_{-2},m_{-3})$$ is a branched double cover $G\aq T_{\C} \to \C^6$.  It extends to a continuous injection $\mathcal{M}\times m_{4}$ into $\C^7$.

We can show that the first six invariants are transpose invariant and that the transpose defines an involution which switches the values of $m_4$ with respect to the roots of the relation above.

Since the determinant is 1, $m_k(X_2^{-1})=m_{-k}(X_2)$, and so when $X_2\in K$ we have $\overline{m_k(X_2)}=m_k(X_2^{-1})=m_{-k}(X_2)$.  So we can ``realify'' these coordinates, as we have done earlier, by $um_k=(m_k+m_{-k})/2$ and $um_{-k}=(m_k-m_{-k})/2i$.  Consequently we have a branched double cover $$\mathcal{UM}=(um_1,um_2,um_3,um_{-1},um_{-2},um_{-3}):K/T\to D_2^{\times 3}\subset \mathbb{R}^6, $$ over its image where $D_2$ is the closed real disc of radius 1 with center 0; also the mapping $\mathcal{UM}\times m_4$ gives a homeomorphism of $K/T$ onto its image in $\mathbb{R}^6\times \C$. We will eventually show this image is homeomorphic to a real six sphere, and the image of $\mathcal{UM}$ is a closed real 6 ball.  This will follow from our theorems, although it is not now apparent.

\end{rem}

\subsubsection{Contractible parts.}
Recall that a topological space $S$ is said to be $k$-connected, if the homotopy
groups $\pi_i(S)=0$ for $0\leq i\leq k$. The property of a topological manifold to be contractible can be
expressed in terms of its homotopy groups:  An $n$-dimensional topological manifold is contractible if and only if it is $n$-connected.

We now apply our main result to prove the following lemma.

\begin{lem}\label{contractiblelemma}
Both $B_+$ and $B_-$ are contractible.
\end{lem}
\begin{proof}
Let $\mathsf{t}:\X_2(G)\to\X_2(G)$ be the transpose involution.  If $\overline{B_+}$ is contractible, then so is $B_+=(\overline{B_+})^o$ and $B_-=\mathsf{t}(B_+)$.  This involution defines a $\mathbb{Z}_2$ group action $\langle \mathsf{t} \rangle$.  We have already seen that $\X_2(G)/\langle \mathsf{t} \rangle\cong \C^8$ and $\X_2(K)/\langle \mathsf{t} \rangle\cong \overline{B_+}$, since the branching locus is exactly the transpose fixed elements in both cases.

For all $t\in [0,1]$ and any $[[\rho]]_G\in \X_2(G)$, we will show $$\Phi_t\left([[\rho]]_G^{\mathsf{t}}\right)=\big(\Phi_t\left([[\rho]]_G\right)\big)^{\mathsf{t}}.$$
Given any $[[\rho]]_G$ there exists a representative $\rho=(g_1,g_2)$ so $g_1^{\mathsf{t}}=g_1$.  This follows since the semi-simplification of any such representative includes a diagonal element for its first component. In particular, with respect to the polar decomposition $g_1=d_1d_2$ where both $d_k$ are diagonal and $d_1$ is unitary, and $g_2=ke^p$.  Now, $\rho^{\mathsf{t}}\sim (k^{\mathsf{t}}d_1\overline{k}k^{\mathsf{t}}d_2 \overline{k},k^{\mathsf{t}}e^{-p})$ where $\sim$ means conjugate.  Thus 
\begin{align*}
\Phi_t\left([[\rho]]_G^{\mathsf{t}}\right)&= \left[\left[\left(k^{\mathsf{t}}d_1\overline{k}\left(k^{\mathsf{t}}d_2 \overline{k}\right)^{1-t},k^{\mathsf{t}}e^{-(1-t)p}\right)\right]\right]_G\\
&=\left[\left[\left(k^{\mathsf{t}}d_1\overline{k}k^{\mathsf{t}}d_2^{1-t} \overline{k},k^{\mathsf{t}}e^{-(1-t)p}\right)\right]\right]_G\\
&=\left[\left[\left(k^{\mathsf{t}}d_1d_2^{1-t}\overline{k},k^{\mathsf{t}}e^{-(1-t)p}\right)\right]\right]_G.
\end{align*}

On the other hand, $$\left(\Phi_t\left([[\rho]]_G\right)\right)^{\mathsf{t}}=\left([[(d_1d_2^{1-t},ke^{(1-t)p})]]_G\right)^{\mathsf{t}}=\left[\left[\left(k^{\mathsf{t}}d_1d_2^{1-t}\overline{k},k^{\mathsf{t}}e^{-(1-t)p}\right)\right]\right]_G.$$

Therefore, the deformation retraction of $\X_r(G)$ to $\X_r(K)$ descends to a deformation retraction of $\X_2(G)/\langle \mathsf{t} \rangle$ to $\X_2(K)/\langle \mathsf{t} \rangle$ by Proposition \ref{Kequivariantdeformation}.  Thus the homotopy groups of $\C^8$ and $B_+$ are the same; that is, $B_+$ has trivial homotopy groups (8-connected).  In turn this implies $B_+$ is contractible since $B_+$ is a topological manifold of real dimension 8.
\end{proof}

Let $p\in (\X_1)^o$ and define $B_{\pm}(K/T)=\big(\{p\}\times K/T\big)\cap B_{\pm}$.  Up to homeomorphism, these sets are independent of the choice of point $p$.  Also, $B_{+}(K/T)\sqcup B_{-}(K/T)\cong S$.

\begin{cor}
The ``upper hemisphere'' $B_{+}(K/T)$ and the ``lower hemisphere'' $B_{-}(K/T)$ are disjoint contractible subspaces of $K/T$.
\end{cor}

\begin{proof}
By definition they are disjoint subspaces.  Moreover, each is the base space of a trivial fibration with contractible total space (from Lemma \ref{contractiblelemma}) and contractible fiber (interior of the alcove).  Hence the base space must be contractible as well (long exact sequence in homotopy).
\end{proof}

\begin{rem}
The transpose mapping also defines an involution on $\X_3(\SLm{2})$, which is likewise equivariant with respect to the deformation retraction.  Since  the transpose is ``trivial'' on $\X_2(\SLm{2})$ (that is, there is a transposition invariant representative in every extended orbit, see \cite{Fl} for details), this somewhat explains why one needs three $\SLm{2}$ matrices to get a double cover and only two $\SLm{3}$ matrices to get a similar structure. 
\end{rem}

\subsubsection{The topology of the decomposition.}
A topological space $M$ is said to have a \emph{non-trivial product decomposition}, if $M\cong M_1\times M_2$ where each factor is not a point.

\begin{prop}\label{luftlemma}
If an open real topological $n$-manifold $M$ is contractible for $n\geq 5$, and there exists non-trivial spaces $X$ and $Y$ so $M\cong X\times Y$, then $M\cong \mathbb{R}^n$.
\end{prop}

For a proof of this proposition see \cite{Lu} and references therein.

We now prove
\begin{lem}
$B_+$ and $B_-$ are homeomorphic to open real 8-balls; that is, are open cells.
\end{lem}

\begin{proof}
From Lemma \ref{contractiblelemma}, both $B_+$ and $B_-$ are contractible and from Proposition \ref{productdecomp} they each have non-trivial product decompositions.

Then  $B_+$ is homeomorphic to $\mathbb{R}^8$, by Proposition \ref{luftlemma}.

Therefore, reversing the homeomorphism $h:B_+\to\mathbb{R}^8$ we conclude that $B_+$ (and thus $B_-$) is a bounded open 8-cell.
\end{proof}

\begin{lem}
$B_0\cong S^7$.
\end{lem}
\begin{proof}
We note a technical fact:  since $\X_r(K)-B_-=B_+\sqcup B_0$ is closed, we have $\overline{B_+}-B_+=B_0$; that is $\partial B_+=B_0=\partial B_-$.  

We wish to conclude that $B_0$ is $S^7$.  For this to be true, it is essential that $B_+\sqcup B_0$ injects into $\mathbb{R}^8$.   If this was not the case we could not conclude that the boundary $B_0$ is a sphere: think of the one point gluing of two 2-spheres.  The complement of the boundary is two bounded open 2-cells but they do not give a 2-sphere.  What goes wrong is the 2-cell and its boundary do not inject into the plane.  However, $\mathcal{U}$ is a continuous map on $B_+$ that extends continuously to its boundary to give an injection (embedding).  

Moreover, it is another technical fact that since the mapping $\X_2(K)\hookrightarrow \mathbb{R}^9$ is a homeomorphism onto its image, that $B_+\sqcup B_0$ is homeomorphic onto its image which itself homeomorphically projects onto its image in $\mathbb{R}^8$ (since the restriction of the projection is injective and proper).

Thus since each $B_+$ and $B_-$ are open bounded real 8 balls given by $h:B_+ \to \mathbb{R}^8$, and the surjective boundary mapping $B_0 \to \partial\big( h(B_+)\big)\subset \mathbb{R}^8$ is also injective, it follows that $B_0\cong S^7$.
\end{proof}

\subsubsection{Proof of Theorem \ref{s8thm}}

\begin{proof}
Both $B_+$ and $B_-$ are homeomorphic to 8-balls and $B_+$ and $B_-$ are glued together at their common boundary $B_0\cong S^7$ and thus gives an 8-sphere.  Since $\X_2(\SUm{3})=B_+\sqcup B_0\sqcup B_-$, it follows that it is an 8-sphere as well.
\end{proof}

\begin{cor}\label{s6corollary}
$K/T\cong S^6$
\end{cor}

\begin{proof}
Fixing two coordinates in an eight sphere gives a six sphere.  Consequently, 
since $(\X_1)^o\times K/T\subset S^8$ fixing any point $p\in (\X_1)^o$ gives a homeomorphic copy of $K/T$ which is thus $S^6$.  Moreover, the hemispheres $B_{\pm}(K/T)\subset K/T$ are each open 6 balls (a priori not clear).

\end{proof}

\begin{rem}
A direct proof of Corollary \ref{s6corollary} using constructive invariant theory would come from showing the upper and lower hemispheres are contractible and then using our explicit descriptions of $\C[G/T_{\C}]$ and $\mathbb{R}[K/T]$ to prove the boundary is $S^5$ $($this is computationally non-trivial$)$.  Also one could analyze the formula for the branching locus of the ``minors mapping'' over the triple product of the real closed 2-disc $D_2$ $($again this is non-trivial computationally$)$.  

Fortunately, Theorem \ref{s8thm} allows us to avoid these computations.  Our proof of Corollary \ref{s6corollary} is an example of what may be called ``topological invariant theory''; that is, the study which endeavors to relate the topologies of quotients of $G$-varieties to the topology of real semi-algebraic $K$-quotients via constructive invariant-theoretic methods. 
\end{rem}

\begin{rem}
One may naturally ask if the balls $B_+$ and $B_-$ are perhaps star-shaped in the trace coordinates, as are the corresponding subsets in $\X_3=\R_3(\SUm{2})/\SUm{2}$  $($from the depiction of $\X_1(\SLm{3})$ and the partial decomposition, they are clearly not convex$)$.  To answer the question in the case of $\X_3$ we used a complete description of its semi-algebraic structure, and a slice. In the case of $\R_2(\SUm{3})/\SUm{3}$ we do not know if our semi-algebraic description is complete, and do not have a slice.  %However, numerical experimentation has shown that the semi-algebraic set we described $($in trace coordinates$)$ is not star-shaped.  So either our description of the semi-algebraic structure is incomplete or $\R_2(\SUm{3})/\SUm{3}$ is not star shaped, we cannot have both properties.  
Experimentation with the decomposition $B_+ \cong \X_1^o\times B_+(K/T)$ suggests that the sets $B_+$ and $B_-$ are star-shaped.  
%So we believe the semi-algebraic structure $\eqref{semialgset}$ given in Section \ref{semialgebraic} is incomplete.  
Such questions concern the classical geometry of these spaces, clearly dependent on one's choice of global coordinates.
\end{rem}

\subsection{Beyond $(r,n)=(r,1),(1,n),(2,2),(2,3),(3,2)$.}
We claim the other cases of $\X_r(K)$ not considered above have non-manifold singularities.

From \cite{BC}, we know that the rank greater than or equal to 4 cases of $\SUm{2}$ all have complex projective cone singularities.  Since $\SUm{2}\hookrightarrow \SUm{n}$ for $n\geq2$ and preserves a representation being abelian, one would then expect the same result (with perhaps worse singularities, but not better) for all $\X_r(K)$ with rank greater than 3.

This leads to the following conjecture which we address (with other topics) in a coming paper:

\begin{conj}
The only cases where $\X_r(\SUm{n})$ is a topological manifold are $(r,n)=(r,1),(1,n),(2,2),(2,3),(3,2)$.
\end{conj}

\section{Some remarks on groups other than $\F_r$}

As mentioned in the introduction, the main theorem does not generalize from free groups to arbitrary finitely generated groups. We now give examples of groups other than $\F_r$ where Theorem \ref{thm:main} fails, and examples where it holds. We let $\zeta(G)$ denote the center of a group $G$.

It would be interesting to obtain group theoretic conditions on an arbitrary finitely generated group $\Gamma$ that would ensure that $\X_{\Gamma}(G)=\hm(\Gamma,G)\aq G$ and $\X_{\Gamma}(K)=\hm(\Gamma,K)/K$ are homotopy equivalent, or even determine the values of $k$ for which the homotopy groups $\pi_k(\X_{\Gamma}(G))$ and $\pi_k(\X_{\Gamma}(K))$ are isomorphic.

\subsection{Surface groups and holomorphic vector bundles}

Let $\Gamma_{g}=\langle\xt_{1},\yt_{1},...,\xt_{g},\yt_{g},\mathtt{p}\ |\ \Pi_{i=1}^{g}[\xt_{i},\yt_{i}]=\mathtt{p}\rangle$,
where $\mathtt{p}$ is central. In these terms define \[
\X_{\Gamma_{g}}^{d}(\SLm{n})=\hm_{d}(\Gamma_{g},\SLm{n})\aq\SLm{n},\]
 where $\hm_{d}(\Gamma_{g},\SLm{n})$ is the set of all homomorphisms
that map the central element $\mathtt{p}$ to a fixed element $d\in\mathbb{Z}_{n}\cong\zeta(\SLm{n})$.
This is the $\SLm{n}$-\emph{twisted character variety} of $\pi_{1}(\Sigma_{g})$
where $\Sigma_{g}$ is a genus $g$ closed surface. We note that topologically,
$\Sigma_{g}$ is a connected sum of $g$ tori and its fundamental
group is $\Gamma_{g}$ when $\mathtt{p}=1$, and in fact 
$\X_{\Gamma_{g}}^{0}(\SLm{n})=\X_{\pi_{1}(\Sigma_{g})}(\SLm{n})$, for $d=0$.

Each $\X_{\Gamma_{g}}^{d}(\SLm{n})$ is a component of the quotient
\[
\hm(\Gamma_{g},\SLm{n})\aq\SLm{n},\]
 and is smooth if $n$ and $d$ are coprime since in these cases the
representations are irreducible. Replacing $\SLm{n}$ by $\mathrm{SU}(n)$
we have the related spaces $\X_{\Gamma_{g}}^{d}(\mathrm{SU}(n))=\hm_{d}(\Gamma_{g},\mathrm{SU}(n))/\mathrm{SU}(n)$,
which we refer to as the $\mathrm{SU}(n)$-\emph{twisted character
variety} of $\pi_{1}(\Sigma_{g})$.

By endowing $\Sigma_{g}$ with a complex structure (so $\Sigma_{g}$
becomes a Riemann surface), one can consider holomorphic vector bundles
over it $E\to\Sigma_{g}$; these are topologically classified by two
invariants, their dimension or rank $r$ and their degree $d$, which
is the number of zeros minus poles (with multiplicities) of a meromorphic
section of its determinant, or top exterior power, $\det E=\bigwedge^{r}E$.
It can be proved that the space of all such vector bundles over $\Sigma_{g}$,
of rank $n$, degree $d$, and such that $n$ and $d$ are coprime and
$\det E$ is a fixed line bundle, denoted $\mathcal{N}_{n,d}$, carries
a natural structure of smooth projective variety of dimension $n^{2}(g-1)+1$.
Moreover, Narasimhan and Seshadri \cite{NS} showed that there is
a homeomorphism \[
\mathcal{N}_{n,d}\cong\hm_{d}(\Gamma_{g},\mathrm{SU}(n))/\mathrm{SU}(n).\]

The Poincar\'e polynomials of $\mathcal{N}_{n,d}$ have been computed.
In particular, for genus $g=2$ and odd degree (see \cite{DR} or
\cite{AB}), we have: \[
P_{t}(\mathcal{N}_{2,d})=1+t^{2}+4t^{3}+t^{4}+t^{6}.\]

Another natural holomorphic object associated with $\Sigma_{g}$ is
the moduli space of Higgs bundles over $\Sigma_{g}$ (see \cite{Hi}).
These are pairs $(E,\phi)$ consisting of a holomorphic vector bundle
$E\to\Sigma_{g}$ together with a global section of $\End E$-valued
holomorphic one forms, $\phi\in H^{0}(\Sigma_{g},\End E\otimes\Omega^{1})$,
where $\Omega^{1}$ denotes the canonical bundle of $\Sigma_{g}$
(the holomorphic cotangent bundle). Using infinite dimensional methods
of symplectic and holomorphic geometry, it can be showed that the
moduli space of (polystable) Higgs bundles of rank $n$, coprime degree
$d$ and fixed determinant, denoted $\mathcal{M}_{n,d}$, is a connected
smooth hyperk\"ahler manifold of dimension $n^{2}(2g-2)+2.$ Moreover,
for any coprime $n$ and $d$: \[
\mathcal{M}_{n,d}\cong\X_{\Gamma_{g}}^{d}(\SLm{n}).\]
We note that each $\mathcal{M}_{n,d}$ is generally a component of
$$\hm(\Gamma_{g},\SLm{n})\aq\SLm{n},$$ and we have for any integer
$k$, $\mathcal{M}_{n,nk}\cong\mathcal{M}_{n,0}\cong\X_{\pi_{1}(\Sigma_{g})}(\SLm{n})$,
which is the case when these moduli spaces are singular. Hitchin (see
\cite{Hi}) also computed the Poincar\'e polynomial of $\mathcal{M}_{2,d}$,
which is for $g=2$ and odd $d$ given by \[
P_{t}(\mathcal{M}_{n,d})=1+t^{2}+4t^{3}+2t^{4}+34t^{5}+2t^{6}.\]
See \cite{BGG} and \cite{HaRo} for background, constructions, definitions,
and theorems.

Comparing the Poincar\'e polynomials of these two corresponding components,
one concludes that they cannot be homotopic, let alone deformation
retracts of each other.

Any deformation retraction between $\hm_{d}(\Gamma_{g},\SLm{2})\aq\SLm{2}$
and $\hm_{d}(\Gamma_{g},\mathrm{SU}(2))/\mathrm{SU}(2)$ would certainly
restrict to components and thus each corresponding component would
be homotopy equivalent, which we just showed was not the case in general
(for the components corresponding to odd degree in the $g=2$ example).

Recently, Daskalopoulos and Wentworth \cite{DW} have shown, by explicit computation, that the Poincar\'e polynomials of $\X_{\pi_{1}(\Sigma_{g})}(\mathrm{SU}(2))$ and $\X_{\pi_{1}(\Sigma_{g})}(\SLm{2})$
(respectively, the singular moduli spaces of rank 2 trivial determinant 
vector bundles and Higgs bundles) are not the same.
Therefore, this provides a more direct example of a finitely generated group
$\pi_{1}(\Sigma_{g})$ for which our main theorem does not apply, in the case $G=\SLm{2}$ and $K=\mathrm{SU}(2)$.

\subsection{Free abelian groups}
On the other hand, a natural example to consider is the case of $g=1$ and $d=0$ or more generally abelian representations.  These prove to be non-counterexamples, as described below.

Let $\pi$ be a free abelian group of rank $r$, so that it has the
presentation\[
\pi=\left\langle \xt_{1},...,\xt_{r}\ |\ \xt_{j}\xt_{k}\xt_{j}^{-1}\xt_{k}^{-1}=1\quad\mbox{for all }1\leq j<k\leq r\right\rangle .\]
Let $K$ be a compact linear Lie group $K$ and $G\subset \mathrm{GL}(n,\mathbb{C})$
be its complexification. Then $\hm(\pi,K)$ can be identified with
the subset of $\R_r(K)$ consisting of $r$-tuples of matrices
that pairwise commute, and similarly for $\hm(\pi,G)$. We note that although $\X_{\pi}(G)=\hm(\pi,G)\aq G=\mathrm{Spec}(\C[\hm(\pi,G)]^G)$ is a potentially non-reduced affine scheme, we only consider its complex points, corresponding to $\mathrm{Spec}_{max}(\C[\X(\pi,G)]\cong \mathrm{Spec}_{max}(\C[\X(\pi,G)]/\sqrt{0})$, in what follows.

\begin{prop}
There is a strong deformation retraction of $\hm(\pi,G)\aq G$ onto $\hm(\pi,K)/K$.
\end{prop}
\begin{proof}

It is well known that any commuting collection of elements of a compact
Lie group $K$ lie on a fixed maximal torus $T\subset K$. Without loss of generality, we can assume $T$ consists of the diagonal matrices in $K$.  Moreover,
the conjugation action on such a torus is reduced to the action of
the Weyl group $W$. Therefore, we have
\[
\hm(\pi,K)/K\cong T^{\times r}/W.
\]

It is also a standard fact that any commuting collection of elements
of $G\subset \GLm{n}$ can be simultaneously put in upper
triangular form. Let $D\subset G$ be the maximal torus that contains $T$, on which $W$ acts. Consequently, $D$ consists of diagonal matrices as well.  We have:
$$
\hm(\pi,G)\aq G\cong D^{\times r}/W,
$$
since each upper triangular representation can be conjugate invariantly deformed to a diagonal one; thus they are in the orbit closure.

From the above identification, we clearly see that $\hm(\pi,G)\aq G$ deformation retracts to $\hm(\pi,K)/K$. Indeed, $D^{\times r}/W$ deformation retracts to $T^{\times r}/W$ by Proposition \ref{Kequivariantdeformation}, since the deformation retraction $(\C^*)^N\cong D\to T\cong (S^1)^N$ is $W$-equivariant where $N$ is the dimension of $D$.  Note that this coincides with the deformation retraction $G\to K$ upon restriction to the respective maximal tori.
\end{proof}

\section*{Appendix}

We now prove Lemma \ref{unitarylemma}:  If two $K$-valued free group representations are conjugate by an element in $G$, then they
are conjugate by an element in $K$. 

\begin{proof}

In \cite{DK} (page 152), it shown that any element of $K$ is $K$-conjugate to an element of a given maximal torus (this action is transitive on tori).  For matrix groups one such torus consists of diagonal elements (and all are $K$-conjugate).  Thus, any element in $K$ is $K$-conjugate to a diagonal matrix.

Now consider $\psi_{1}=(k_{1},..,k_{r})$ and $\psi_{2}=(k'_{1},..,k'_{r})$
in $\R_r(K)$ and suppose they are conjugate by an element in
$G$. Then $k_{1}$ and $k'_{1}$ have the same normal forms, in this
case diagonal.  Moreover, in general the action of the Weyl group can be realized by elements in $K$ from the normalizer of a fixed torus:  $N_K(T)/T$.  Thus permuting diagonal elements is likewise achievable by the $K$-conjugation action.  

So there exists $u$ and $u'$ in $K$ so $uk_{1}u^{-1}=u'k'_{1}(u')^{-1}$ since permutations of diagonal
entries of diagonal matrices is given by $K$-conjugation. Hence $k_{1}=u^{-1}u'k'_{1}(u')^{-1}u$, a 
$K$-conjugate. This implies that we may simultaneously diagonalize the
first entries of $\psi_{1}$ and a $K$-conjugate of $\psi_{2}$ using
a single element in $K$. First replace $(k'_{1},...,k'_{r})$ by its
conjugate by $u^{-1}u'$ (this only changes the representative of
$\psi_{2}$ in $[[\psi_{2}]]_K$), and then since the first
entries in the $r$-tuples are now equal we diagonalize by another element in $K$ (in this case $u$).

Now we claim that the $r$-tuples are in fact equal, or else there
is a permutation matrix (necessarily in $K$) that will make them
equal. After reconjugating by $u$, let $d_{1}$ be the diagonalization
of $k_{1}$. Since the $r$-tuples are $G$-conjugate, there exists
$g$ in $G$ so 
$$(d_{1},uk_{2}u^{-1},...,uk_{r}u^{-1})=(gd_{1}g^{-1},gu'k'_{2}u'^{-1}g^{-1},...,gu'k'_{r}u'^{-1}g^{-1}),$$
but then $d_{1}=gd_{1}g^{-1}$ is diagonal. Hence, $g$ is necessarily
a matrix that permutes the diagonal entries since conjugation preserves
the eigenvalues, and all such permutation matrices (which preserve the give torus) are in $K$. Therefore,
the tuples are conjugate by an element of $K$, as was to be shown.
Explicitly, we have: $$
\psi_{1}=(u^{-1}gu')\psi_{2}(u^{-1}gu')^{-1},$$
where $u'$ normalizes $k_{1}'$, $u$ normalizes $k_{1}$, and $g$
is a permutation matrix which is not the identity only if there are
repeated eigenvalues. 
\end{proof}

%\bibliography{bib}
\def\cdprime{$''$} \def\cdprime{$''$} \def\cprime{$'$} \def\cprime{$'$}
  \def\cprime{$'$} \def\cprime{$'$}
\providecommand{\bysame}{\leavevmode\hbox to3em{\hrulefill}\thinspace}
\providecommand{\MR}{\relax\ifhmode\unskip\space\fi MR }
% \MRhref is called by the amsart/book/proc definition of \MR.
\providecommand{\MRhref}[2]{%
  \href{http://www.ams.org/mathscinet-getitem?mr=#1}{#2}
}
\providecommand{\href}[2]{#2}

\end{document}